\documentclass[a4paper,10pt]{article}
\usepackage[utf8]{inputenc}
\usepackage[T1]{fontenc}
\usepackage{lmodern}
\usepackage{amsfonts}
\usepackage{a4wide}
\usepackage{geometry}
\usepackage{graphicx}
\usepackage{amssymb, amsmath, amsopn, amstext, amscd, amsthm}
\usepackage{color}
\usepackage{version}
\usepackage{todonotes}
\usepackage{xcolor}

\newtheorem{definition}{Definition}
\newtheorem{proposition}{Proposition}
\newtheorem{theorem}{Theorem}
\newtheorem{lm}{Lemma}
\newtheorem{remark}{Remark}

\DeclareMathOperator*{\argmin}{argmin}
\DeclareMathOperator*{\argmax}{argmax}

\def\argmin{\mathop{\rm argmin}}

\newcommand{\1}{{\bf 1}}
\newcommand{\card}{{\rm card}}

\begin{document}

\title{Confidence sets with expected sizes for\\ 
Multiclass Classification}
 \author{Christophe Denis\footnote{Christophe.Denis@u-pem.fr}  \, and Mohamed Hebiri\footnote{Mohamed.Hebiri@u-pem.fr}\\
\small{ LAMA, UMR-CNRS 8050,}\\
\small{Universit\'{e} Paris Est -- Marne-la-Vall\'ee}
}
 \date{}
\maketitle

\begin{abstract}
Multiclass classification problems such as image annotation can involve a large number
of classes. In this context, confusion between classes can occur, and single label classification may be misleading.
We provide in the present paper a general device that,
given an unlabeled dataset and a score function defined as the minimizer of some empirical and convex risk,
outputs a set of class labels, instead of a single one.
Interestingly, this procedure does not require that the unlabeled dataset explores the whole classes.
Even more, the method is calibrated to control the expected size of the output set while minimizing the classification risk.
We show the statistical optimality of the procedure and establish rates of convergence under the Tsybakov margin condition.
It turns out that these rates are linear on the number of labels.
We apply our methodology to convex aggregation of confidence sets based on the $V$-fold cross validation principle also known as the superlearning principle~\cite{vanderLaanSL}. 
We illustrate the numerical performance of the procedure on real data and demonstrate in particular that with moderate expected size, w.r.t. the number of labels, the procedure provides significant improvement of the classification risk.
\\
 \vspace*{0.25cm} 

\vspace*{-0.3cm} 

\qquad\qquad  cumulative distribution functions, convex loss, superlearning.

\end{abstract}

%%%%%%%%%%%%%%%%%%%%%%%%%%%%%%%%%%%%%%%%%
%%%%%%%%%%%%%%%%%%%%%%%%%%%%%%%%%%%%%%%%%
\section{Introduction}
\label{Sec:Intro}
%%%%%%%%%%%%%%%%%%%%%%%%%%%%%%%%%%%%%%%%%
%%%%%%%%%%%%%%%%%%%%%%%%%%%%%%%%%%%%%%%%%

The advent of high-throughput technology has generated tremendous amounts of large and high-dimensional classification data. This allows classification at unprecedented 
scales with hundreds or even more classes. The standard approach to classification in the multiclass setting is to use a classification rule for assigning a single label. More specifically, it consists in assigning a single label $Y \in \mathcal{Y}$, with $ \mathcal{Y}= \{1, \ldots, K\}$, to a given input example $X \in \mathcal{X}$ among a 
collection of labels. However, while a large number of classes can lead to precise characterizations of data points, similarities among classes also bear the risk of 
confusion and misclassification. Hence, assigning a single label can lead to wrong or ambiguous results.  

In this paper, we address this problem by introducing an approach that yields sets of labels as outputs, namely, confidence sets.
A confidence set $\Gamma$ is a function that maps $\mathcal{X}$ onto $2^{\mathcal{Y}}$.
A natural way to obtain a set of labels is to use 
ranked outputs of the classification rule.  For example, one could take the classes that correspond to the top-level conditional probabilities $\mathbb{P}(Y=\cdot|X =x)$. 
Here, we provide a more general approach where we control the {\it expected} size of the confidence sets. For a confidence set $\Gamma$, the expected size of $\Gamma$ is defined as
$\mathbb{E}[|\Gamma(X)|]$, where $|\cdot|$ stands for the cardinality.
For a sample $X$ and  given  an expected set size $\beta$,
we provide an algorithm that outputs a set $\hat{\Gamma}(X)$ such that $\mathbb{E}[|\hat{\Gamma}(X)|] \approx \beta$.
Furthermore, the procedure aims at minimizing the classification error given by
\begin{equation*}
\mathbb{P}\left(Y \notin \hat{\Gamma}(X)\right) \approx   \min_{{\scriptscriptstyle\Gamma \, : \, \mathbb{E}[|\Gamma(X)|] = \beta}} \mathbb{P}\left(Y \notin 
\Gamma(X)\right)  = \mathcal{R}^{*}_{\beta}.
\end{equation*}
We establish a close formula of the oracle confidence set~$\Gamma^* =  \argmin_{{\scriptscriptstyle\Gamma \, : \, \mathbb{E}[|\Gamma(X)|] = \beta}} \mathbb{P}\left(Y \notin  \Gamma(X)\right)$ that involves the cumulative distribution functions of the conditional probabilities. 
Besides, we formulate a data-driven counterpart of $\Gamma^{*}$ based on the minimization of the empirical risk.
However, the natural risk function in the multiclass setting is non convex, and then minimizing it is proving to be a computational issue.
As a remedy, convex surrogates are often used in machine learning, and more specifically in classification as
reflected by the popularity of methods such as Boosting~\cite{Freund_Schapire97}, logistic regression~\cite{Friedman_Hastie_Tibshirani00}
and support vector machine~\cite{VapnikLivre}.
In our problem this translates by considering some convex surrogate of the $0/1$-loss in $\mathbb{P}\left(Y \notin \Gamma(X)\right)$; 
we introduce a new convex risk function which enables us to emphasize specific aspects of confidence sets.  
That convex risk function is partly inspired by the works in~\cite{Zhang04, BartlettJordanMcauliffe06, YW10} that deal with binary classification.

Our approach displays two main features. 
First, our method can be implemented in a semi-supervised way [Vap98].
More precisely, the method is a two-steps procedure that requires the estimation of score functions (as minimizers of some convex risk function) in a first step and the estimation of the cumulative distribution of these scores in a second one. Hence, the first step requires labeled data whereas the second one involves only unlabeled samples. 
Notably, the unlabeled sample does not necessary consists of examples coming from the whole classes. This aspects is fondamental when we deal with a large number of classes, some of which been sparsely represented. 
Second, from the theoretical point of view, we provide an oracle inequality satisfied by $\hat{\Gamma}$,
the empirical confidence set that results from the minimization of an empirical convex risk. 
The obtained oracle inequality shall enable us to derive rates of convergence
on a particular class of confidence sets under the Tsybakov noise assumption on the data generating distribution.
These rates are linear in the number of classes $K$ and also depend on the regularity of the cumulative distribution functions of the conditional probabilities.

An obvious benefit of considering convex risk minimization is when we deal with aggregation. However, aggregating confidence sets is no simple matter. Another contribution of the present paper is about applying the above methodology to aggregation of confidence sets. More specifically, we provide a generalization of the superlearning algorithm~\cite{vanderLaanSL}
initially introduced in the context of regression and binary classification.
This algorithm relies on the $V$-fold cross-validation principle. We prove the consistency of this aggregation procedure and illustrate 
its relevance on real datasets.
Let us point out that any arbitrary library of machine learning algorithms may be used in this aggregation procedure; we propose in the present paper to exploit support vector machines, random forest procedures or softmax regression since these methods are popular in machine learning and that each of them is associated to a different shape.

We end up this discussion by highlighting in two words what we perceive as being the main contributions of the present paper. We describe an optimal strategy for building confidence sets in multiclass classification setting,
and we derive a new aggregation procedure for confidence sets that still allows controlling the expected size of the resulting confidence set.

\bigskip

\noindent {\it Related works:} 
The closest learning task to the present work is {\it classification with reject option} which is a particular setting in binary classification.
Several papers fall within the scope of this area~\cite{Ch70,HW06,YW10,WY11,LeiConfTheori,DH_ConfSet_15} and differ from each other by the goal they consider. 
Among these references, our procedure is partially inspired by the paper~\cite{DH_ConfSet_15} that also considers a semi-supervised approach to build confidence sets invoking some cumulative distribution functions (of the conditional probabilities themselves in their case).
The similarity is however limited to the definition of oracle confidence sets, the oracle confidence in the present paper being an extension of the one defined in~\cite{DH_ConfSet_15} to the multiclass setting. On the other hand, all the data-driven considerations are completely different (in particular, \cite{DH_ConfSet_15} focuses on plug-in rules) and importantly, we develop here new probabilistic results on sums of cumulative distribution functions of random variables, that are of own interest.

Assigning a set of labels instead of a single one for an input example is not new~\cite{Vovk_livre,WLW04,CDB09,LRW13,CCB_16}. 
One of the most popular methods is based on {\it Conformal Prediction} approach~\cite{Vovk99IntroCP,Vovk02AsymptVal,Vovk_livre}.
In the multiclass classification framework, the goal of this algorithm is to build the smallest set
%\footnote{Controlling the size of the set is not completely mentioned in  the algorithm but can be done in the same time. We refer to the very interesting statistical study of {\it Conformal Predictor} in the binary case in the paper ..., where the authors }
of labels such that its classification error is below a pre-specified level.
Since our procedure aims at minimizing the classification error while keeping under control the size of the set, {\it Conformal Prediction} can be seen as a dual of our method.
It is worth mentioning that conformal predictors approaches need two labeled datasets where we only need one labeled dataset, the second being unlabeled.
We refer to the very interesting statistical study of {\it Conformal Predictors} in the binary case in the paper~\cite{LeiConfTheori}.

\bigskip

\noindent {\it Notation:} First, we state general notation. Let $\mathcal{Y} = \{1,\ldots,K\}$, 
with $K \geq 2$ being an integer.
Let $(X ,Y)$  be the generic data-structure taking its values in $\mathcal{X} \times \mathcal{Y}$
with distribution $\mathbb{P}$.
The goal in classification is to predict the label $Y$ given an observation of $X$.
This is performed based on a classifier (or classification rule) $s$ which is a function mapping $\mathcal{X}$ onto 
$\mathcal{Y}$. Let $\mathcal{S}$ be the set of all classifiers.
The misclassification risk $R$ associated with $s \in \mathcal{S}$ is defined as
\begin{equation*}
R(s) = \mathbb{P}(s(X) \neq Y).
\end{equation*} 
Moreover, the minimizer of $R$ over $\mathcal{S}$ is the Bayes classifier, denoted by $s^*$, and is characterized by
\begin{equation*}
s^{*}(\cdot) = \argmax_{k \in \mathcal{Y}} p_{k}(\cdot),
\end{equation*}
where $p_{k}(x) = \mathbb{P}(Y = k | X = x)$ for $x \in \mathcal{X}$ and $k \in \mathcal{Y}$.
\\
Let us now consider more specific notation related to the multiclass confidence set setting. 
Let a confidence set be any measurable function that maps $\mathcal{X}$
onto $2^{\mathcal{Y}}$. 
Let $\Gamma$ be a confidence set. This confidence set is characterized by two attributes.
The first one is the risk associated to the confidence set
\begin{equation}
	\label{Sec1:RiskConfidenceSet}
		\mathcal{R}\left(\Gamma\right)  =  \mathbb{P}\left(Y \notin \Gamma(X)\right),
\end{equation}
and is related to its accuracy.
The second attribute is linked to the information given by the confidence set.
It is defined as
\begin{equation}
	\label{Sec1:ExpConfidenceSet}
\mathcal{I}(\Gamma) =  \mathbb{E}\left(|\Gamma(X)|\right),
 \end{equation}
where $|\cdot|$ stands for the cardinality. 
Moreover, for some $\beta \in [1,K]$, we say that, for two confidence sets $\Gamma$ and $\Gamma '$ such 
that 
$\mathcal{I}\left(\Gamma \right)=\mathcal{I}\left(\Gamma '\right) = \beta$, the confidence set $\Gamma$ 
is ``better'' than $\Gamma '$
if $\mathcal{R}\left(\Gamma \right) \leq \mathcal{R}\left(\Gamma '\right)$.

\bigskip

\noindent {\it Organization of the paper:} 
The rest of the paper is organized as follows. 
Next section is devoted to the definition and the main properties
of the oracle confidence set for multiclass classification.
The empirical risk minimization procedure is provided in 
Section~\ref{Sec:EmpRisk}. 
Rates of convergence for the confidence set that results from this minimization can also be found in this section.
We present an application of our procedure to  aggregation of confidence sets in Section~\ref{Sec:numRes}.
We finally draw some conclusions and present perspectives of our work in Section~5. Proofs of
our results are postponed to the Appendix.

\section{Confidence set for multiclass classification}
\label{Sec:confSet}
%%%%%%%%%%%%%%%%%%%%%%%%%%%%%%%%%%%%%%%%%
%%%%%%%%%%%%%%%%%%%%%%%%%%%%%%%%%%%%%%%%%
In the present section, we define a class of confidence sets that are suitable for multiclass classification and referred as {\it Oracle $\beta$-sets}.
For some $\beta \in (0,K)$, these sets are shown to be optimal according to the risk~\eqref{Sec1:RiskConfidenceSet}
with an information~\eqref{Sec1:ExpConfidenceSet} equal to $\beta$. 
Moreover, basic but fondamental properties of Oracle $\beta$-sets can be found in Proposition~\ref{prop:propG},
while Proposition~\ref{prop:prop2} provides another interpretation of these sets.

%%%%%%%%%%%%%%%%%%%%%%%%%%%%%%%%%%%%%%%%%
\subsection{Notation and definition}
\label{subsec:confSetDef}
%%%%%%%%%%%%%%%%%%%%%%%%%%%%%%%%%%%%%%%%%

First of all, we introduce in this section a class of confidence sets that specifies oracle confidence sets.
Let $\beta\in(0,K)$ be a desired information level.
The so-called {\it Oracle $\beta$-sets} are optimal according to the risk~\eqref{Sec1:RiskConfidenceSet} among all the 
confidence sets $\Gamma $ such that $\mathcal{I}(\Gamma)  = \beta$.
Throughout the paper we make the following assumption
{\it 
\begin{itemize}
\item[{\bf (A1)}] For all $k \in \{1,\ldots,K\}$, the cumulative distribution function $F_{p_k}$ of $p_k(X)$ is continuous.
\end{itemize}
}
The definition of the {\it Oracle $\beta$-set} relies on the continuous and decreasing function $G$ defined for any $t\in 
[0,1]$ by
\begin{equation*}
G(t) = \sum_{k=1}^K \bar{F}_{p_k}(t),
\end{equation*}
where for any $k\in \{1,\ldots, K\}$, we denote by $ \bar{F}_{p_k}$ the tail distribution function of $p_k(X)$, that 
is, $ \bar{F}_{p_k}= 1- F_{p_k}$ with $ F_{p_k}$ being the cumulative distribution function (c.d.f.) of $p_k(X)$.
The generalized inverse $G^{-1}$ of $G$ is given by (see \cite{VanderVaartLivre}):
\begin{equation*}
G^{-1}(\beta) = \inf\{t\in [0,1] : \ G(t) \leq \beta \}, \;\; \forall \beta \in (0,K).
\end{equation*}
The functions $G$ and $G^{-1}$ are central in the construction of the Oracle $\beta$-sets. We then provide some of their useful properties in the following proposition.
\begin{proposition}
The following properties on $G$ hold
\label{prop:propG}
\begin{itemize}
\item[$i)$] For every $t \in (0,1)$ and $\beta \in (0,K)$, $G^{-1}(\beta) \leq t \Leftrightarrow \beta \geq G(t)$.
\item[$ii)$]For every $\beta \in (0,K)$, $G(G^{-1}(\beta)) = \beta$.
\item[$iii)$] Let $\varepsilon$ be a random variable, independent of $X$, and distributed from a uniform distribution on 
$\{1, \ldots, K\}$
and let $U$ be uniformly distributed on $[0,K]$.
Define
\begin{equation*}
Z = \sum_{k = 1}^K p_k(X) \1_{\{\varepsilon = k\}}.
\end{equation*}
If the function $G$ is continuous, 
then $G(Z)  \overset{\mathcal{L}}{=} U$ and $G^{-1}(U)  \overset{\mathcal{L}}{=} Z$.  
\end{itemize} 
\end{proposition} 
The proof of Proposition~\ref{prop:propG} relies on Lemma~\ref{lemme:lm1} in the Appendix.
Now, we are able to defined the Oracle $\beta$-set:
\begin{definition}\label{def:OracleSet}
	Let  $\beta \in (0,K)$, the Oracle $\beta$-set is given by
	\begin{eqnarray*}
		\Gamma_\beta^* (X)  & =  & \left\{ k \in \{1,\ldots, K\} : \ G(p_{k}(X)) \leq \beta  \right\} \\
		                   & = & \left\{ k \in \{1,\ldots, K\} : \ p_{k}(X)\geq G^{-1}(\beta)  \right\}.
    \end{eqnarray*}
\end{definition}
This definition of the Oracle $\beta$-set is intuitive and can be related to the binary classification with reject 
option setting~\cite{Ch70,HW06,DH_ConfSet_15} in the following way: a label $k$ is assigned to the Oracle $\beta$-set if the probability $p_k(X)$ is 
large enough. It is worth noting that the function $G$ plays the same role as the c.d.f. of the score 
function in~\cite{DH_ConfSet_15}.
As emphasized by Proposition~\ref{prop:propG}, their introduction allows to control exactly the information~\eqref{Sec1:ExpConfidenceSet}.
Indeed, it follows from the definition of the Oracle $\beta$-set that for each $\beta \in (0,K)$ 
\begin{equation*}
|\Gamma_{\beta}^*(X)| = \sum_{k=1}^K \1_{\{p_k(X) \geq G^{-1}(\beta)\}},
\end{equation*}
and then $\mathcal{I}(\Gamma_\beta^*) = \mathbb{E}\left[  |\Gamma_{\beta}^*(X)|   \right]=G(G^{-1}{(\beta)})$.
Therefore, Proposition~\ref{prop:propG} ensures that 
\begin{equation*}
\mathcal{I}(\Gamma_\beta^*) = \beta.
\end{equation*}
This last display points out that the Oracle $\beta$-sets are indeed $\beta$-level (that is, its information equals $\beta$).
In the next section, we focus on the study of the risk of these oracle confidence sets.

\begin{remark}
Naturally, the definition of Oracle $\beta$-sets can be extended to any $\beta \in  [0,K]$. 
However, the limit cases $\beta=0$ and $\beta=K$ are of limited interest and are completely trivial. 
We then exclude these two limit cases from the present study.
\end{remark}

%%%%%%%%%%%%%%%%%%%%%%%%%%%%%%%%%%%%%%%%%
\subsection{Properties of the oracle confidence sets}
\label{subsec:confSetProp}
%%%%%%%%%%%%%%%%%%%%%%%%%%%%%%%%%%%%%%%%%

Let us first state the optimality of the Oracle $\beta$-set:
\begin{proposition}
	\label{prop:prop1}
	Let Assumption (A1) be satisfied. For any $\beta\in (0,K)$, we have both:
		\begin{enumerate}
			\item The Oracle $\beta$-set $\Gamma_\beta^*$ satisfies the following property:
\begin{equation*}
\mathcal{R}\left( \Gamma_\beta^* \right) =  \inf_{\Gamma } \mathcal{R}\left(\Gamma\right),
\end{equation*}	
where the infimum is taken over all $\beta$-level confidence sets. 
			\item For any $\beta$-level confidence set $\Gamma$, the following holds
\begin{equation}
\label{eq:decompRisk}
0\leq \mathcal{R}\left(\Gamma\right) - \mathcal{R}\left(\Gamma_\beta^*\right)  
=   
\mathbb{E}\left[ \sum_{k\in (\Gamma_\beta^*(X) \, \Delta  \, \Gamma(X))}
\left| p_{k}(X) - G^{-1}(\beta) \right| \right],
\end{equation}	
where the symbol $\Delta$ stands for the symmetric difference of two sets, that is, for two subsets $A$ and $B$ 
of $\{1,\ldots,K\}$, we write $A \, \Delta \,  B = (A \setminus B) \cup (B\setminus A)$. 
\end{enumerate}
\end{proposition}
Several remarks can be made from Proposition~\ref{prop:prop1}. First, for $\beta \in (0,K)$, the 
Oracle $\beta$-set is optimal for the misclassification risk, over the class of $\beta$-level confidence sets.
Moreover, the excess risk of any $\beta$-level confidence set relies on the behavior of the score functions $p_k$ around 
the threshold $G^{-1}(\beta)$.
Finally, we can note that if $K = 2$ and $\beta = 1$, which implies that $G^{-1}(\beta) = 1/2$, Equation~\eqref{eq:decompRisk}  
reduces to the misclassification excess risk in binary classification.
\begin{remark}
\label{rmk:max}
One way to build a confidence set $\Gamma$ with information $\beta$ is to set $\Gamma$ as the $\beta$ top levels conditional probabilities.
In the sequel this method is referred as the $\texttt{max}$ procedure.
This strategy is natural but actually suboptimal as shown by the first point of Proposition~\ref{prop:prop1}. As an illustration, 
we consider a simulation 
scheme with $K = 10$ classes.
We generate $(X,Y)$ according to a mixture model. More precisely,
\begin{itemize}
\item[i)] the label $Y$ is distributed from a uniform distribution on $\{1, \ldots, K\}$;
\item[ii)] conditional on $Y = k$, the feature $X$ is generated according to a multivariate gaussian distribution
with mean parameter $\mu_k \in \mathbb{R}^{10}$ and identity covariance matrix. For each $k =1, \ldots, K$, the vectors $\mu_k$ are i.i.d realizations of uniform distribution on
$[0,4]$.
\end{itemize}
For $\beta = 2$ we evaluate the risks of the Oracle $\beta$-set and the $\texttt{max}$ procedure and obtain respectively
$0.05$  and $0.09$ (with very small variance).  
\end{remark}
\begin{remark}
An important motivation behind the introduction of confidence sets and in particular of Oracle $\beta$-sets is that they might outperform the Bayes rule which can be seen as the Oracle $\beta$-set associated to $\beta=1$. 
This gap in performance is even larger when the number of classes $K$ is large and there is a big confusion between classes.
Such improvement will be illustrated in the numerical study (see Section~\ref{subsec:realData}).
\end{remark}

We end up this section by providing another characterization of the Oracle $\beta$-set.
\begin{proposition}
\label{prop:prop2}
For $t \in [0,1]$, and $\Gamma$ a confidence set, we define
\begin{equation*}
L_t(\Gamma) = \mathbb{P}\left(Y  \notin \Gamma(X)\right) + t \mathcal{I}(\Gamma).
\end{equation*}
For $\beta \in [0,K]$, the following equality holds: 
\begin{equation*}
L_{G^{-1}(\beta)}( \Gamma_\beta^*) = \min_{\Gamma} L_{G^{-1}(\beta)}(\Gamma).
\end{equation*}
\end{proposition}
The proof of this proposition relies on the same arguments as those of Proposition~\ref{prop:prop1}.
It is then omitted.
Proposition~\ref{prop:prop2} states that the Oracle $\beta$-set is defined as the minimizer, over all confidence sets $\Gamma$, of the risk function $L_{t}$ when the tuning 
parameter $t$ is set equal to $G^{-1}(\beta)$.
Note moreover that the risk function $L_t$ is a trade-off, controlled by the parameter $t$, between the risk of a confidence set on the one hand, and the information provided 
by this confidence set on the other hand.
Hence, the risk function $L_t$ can be viewed as a generalization to the multiclass case of the risk function 
provided in~\cite{Ch70, HW06} for binary classification with reject option setting.

%%%%%%%%%%%%%%%%%%%%%%%%%%%%%%%%%%%%%%%%%
%%%%%%%%%%%%%%%%%%%%%%%%%%%%%%%%%%%%%%%%%
\section{Empirical risk minimization}
\label{Sec:EmpRisk}
%%%%%%%%%%%%%%%%%%%%%%%%%%%%%%%%%%%%%%%%%
%%%%%%%%%%%%%%%%%%%%%%%%%%%%%%%%%%%%%%%%%

In this section we introduce and study confidence sets which rely on the
minimization of convex risks. Their definitions and main properties are given in Sections~\ref{subsec:phiRisk}-\ref{subsec:confSetCalib}. 
As a consequence, we deduce a data-driven procedure described in Section~\ref{subsec:dataDriven} with
several theoretical properties, such as rates of convergence, that we demonstrate in Section~\ref{subsec:rateOfConvergence}.

%%%%%%%%%%%%%%%%%%%%%%%%%%%%%%%%%%%%%%%%%
\subsection{$\phi$-risk}
\label{subsec:phiRisk}
%%%%%%%%%%%%%%%%%%%%%%%%%%%%%%%%%%%%%%%%%

Let $f = (f_1, \ldots, f_K):\mathcal{X} \rightarrow \mathbb{R}^{K}$ be a score function and $G_{f}(.) = \sum_{k = 1}^K \bar{F}_{f_{k}}(.)$.
Assuming that the function $G_{f}$ is continuous and given an information level $\beta \in (0,K)$,
there exists $\delta \in \mathbb{R}$, such that $G_{f}(-\delta) = \beta$. Given this simple but important fact, we define
the confidence set $\Gamma_{f,\delta}$ associated to $f$ and $\delta$ as 
\begin{equation}
\label{eq:fconfSet} 
\Gamma_{f,\delta}(X) = \{k \; : \; f_k(X) \geq -\delta\}.
\end{equation}
In this way, the confidence set $\Gamma_{f,\delta}$ consists of top scores, and the threshold $\delta$ is fixed so that $\mathcal{I}\left(\Gamma_{f,\delta}\right) = \beta$.
As a consequence, we naturally aim at solving the problem
\begin{equation*}
\min_{f \in \mathcal{F}}\mathcal{R}(\Gamma_{f,\delta}),
\end{equation*}
where $\mathcal{F}$ is a class of functions.
Due to computational issues, it is convenient to focus on a convex surrogate of the previous minimization problem.
To this end, let $\phi: \mathbb{R} \rightarrow \mathbb{R}$ be a convex function. 
We define the $\phi$-risk of $f$ by
\begin{equation}
\label{eq:phirisk}
R_{\phi}\left(f\right)  =  \mathbb{E}\left[\sum_{k = 1}^K \phi(Z_k f_k(X))\right],
\end{equation}
where $Z_k = 2 \; \1_{\{Y = k\}}-1$ for all $k = 1, \ldots, K$.
Therefore, our target score becomes
\begin{equation*}
\bar{f} \in \argmin_{f \in \mathcal{F}} R_{\phi}\left(f\right),
\end{equation*}
for the purpose of building the confidence $\Gamma_{\bar{f}, \delta}$.
In the sequel, we also introduce $f^*$, the minimizer over the class of all measurable functions, of the $\phi$-risk. 
The notation suppresses the dependence on $\phi$.
It is worth mentioning at this point that the definition of the risk function $R_{\phi}$ is dictated by Equation~\eqref{eq:decompRisk} and suits for confidence sets. Moreover, this function differs from the classical risk function used in the multiclass setting (see~\cite{BartlettTewari}). The reason behind this is that building a confidence set is closer to $K$ binary classification problems.

%%%%%%%%%%%%%%%%%%%%%%%%%%%%%%%%%%%%%%%%%
\subsection{Classification calibration for confidence sets}
\label{subsec:confSetCalib}
%%%%%%%%%%%%%%%%%%%%%%%%%%%%%%%%%%%%%%%%%
Convexification of the risk in classification is a standard technique.
In this section we adapt classical results and tools to confidence sets in the multiclass setting.
We refer the reader to earlier papers as~\cite{Zhang04,BartlettJordanMcauliffe06,YW10} for interesting developments.

One of the main important concept when we deal with convexification of the risk is the notion of calibration of the loss function $\phi$. 
This property  permits to connect confidence sets deduced from the convex risk and from the classification risk.
\begin{definition}
\label{def:calibration}
We say that the function $\phi$ is confidence set calibrated if for all $\beta > 0$, there exists $\delta^{*} \in 
\mathbb{R}$ 
such that 
\begin{equation*}
\Gamma_{f^{*},\delta^{*}} = \Gamma_{\beta}^{*},
\end{equation*}
with $f^*$ being the minimizer of the $\phi$-risk 
\begin{equation*}
f^*  \in \argmin_{f} R_{\phi}\left(f\right),
\end{equation*}
where the infimum is taken over the class of all measurable functions.
Hence, the confidence set based on $f^*$
coincides with the Bayes confidence set. 
\end{definition}
Given this, we can state the following proposition that gives a characterization of the confidence set calibration property in terms of the function $G$.
\begin{proposition}
\label{prop:classCal}
The function $\phi$ is confidence set calibrated if and only if for all $\beta \in (0,K)$, 
there exists $\delta^{*} \in \mathbb{R}$ 
such that
$\phi'(\delta^{*})$ and $\phi'(-\delta^{*})$ both exists, $\phi'(\delta^{*}) <0$, $\phi'(-\delta^{*}) <0$ and
\begin{equation*}
G^{-1}(\beta) = \dfrac{\phi'(\delta^{*})}{\phi'(\delta^{*})+\phi'(-\delta^{*})} \; ,
\end{equation*}
where $\phi'$ denotes the derivative of $\phi$.
\end{proposition}
The proof of the proposition follows the lines of Theorem~1 in~\cite{YW10} with minor modifications.
These characterizations of calibration for confidence sets generalize the notion of calibration in the classification setting as well as the necessary and sufficient condition for the minimizer of the $\phi$-risk to be calibrated. Indeed, if we pick $\delta=0$ in Definition~\ref{def:calibration} and Proposition~\ref{prop:classCal} we exactly come back to the calibration property in the classical classification setting (see~\cite{BartlettJordanMcauliffe06}).
Note that commonly used loss functions as boosting ($x\mapsto \exp(-x)$), least squares ($x \mapsto (x-1)^2$) and logistic ($x \mapsto \log(1+\exp(-x))$)
are examples of calibrated losses (see for instance~\cite{BartlettJordanMcauliffe06,WY11}).
Now, for some score function $f$ and some real number $\delta$ such that $G_{f}(-\delta) = \beta$,  we define the excess risk 
\begin{equation*}
\Delta\mathcal{R}(\Gamma_{f,\delta}) = \mathcal{R}(\Gamma_{f,\delta}) - \mathcal{R}(\Gamma^{*}_{\beta}), 
\end{equation*}
and the excess $\phi$-risk 
\begin{equation*}
\Delta R_{\phi}(f) = R_{\phi}\left(f\right)-R_{\phi}\left(f^{*}\right).
\end{equation*}
We also introduce the marginal conditional excess $\phi$-risk on $f=(f_1,\ldots,f_K)$ as
\begin{equation*}
\Delta R^{k}_{\phi}(f(X)) = p_k(X)(\phi(f_k(X))-\phi(f_k^*(X))) + (1-p_k(X))(\phi(-f_k(X))-\phi(-f_k^*(X))),
\end{equation*} 
for $k=1,\ldots,K$.
The following theorem shows that the consistency in terms of $\phi$-risk implies the consistency in terms of classification risk $\mathcal{R}$.
\begin{theorem}
\label{theo:theo1}
Assume that $\phi$ is confidence set calibrated and assume that there exists constants $C > 0$ and $s \geq 1$ such 
that\footnote{With abuse of notation, we write $ \Delta R_{\phi}^k(-\delta^{*})$ instead of  $\Delta R_{\phi}^k((-\delta^{*},\ldots,-\delta^*))$ since no confusion can occur.}
\begin{equation}
\label{eq:Zhansgg} 
|p_k(X) - G^{-1}(\beta)|^s \leq C \Delta R_{\phi}^k(-\delta^{*}).
\end{equation} 
Let $\hat{f}_n$ be a sequence of scores. We assume that for each $n$, the function $G_{\hat{f}_n}$ is continuous.
Let $\delta_n\in \mathbb{R}$ be such that $G_{\hat{f}_n}(-\delta_n) =\beta$, then
\begin{equation*}
\Delta R_{\phi}(\hat{f}_n) \overset{P}{\rightarrow} 0 \qquad \Rightarrow \qquad \Delta \mathcal{R}(\Gamma_{\hat{f}_n,\delta_n}) 
\overset{P}{\rightarrow} 0.
\end{equation*}
\end{theorem}
The theorem ensures that the convergence in terms of $\phi$ risk implies 
the convergence in terms of risk $\mathcal{R}$. 
This convergence is made possible since we manage, in the proof, to bound the excess risk by (a power of) the excess $\phi$-risk. The assumption needed in this theorem is also standard and is for instance satisfied for the boosting, least square and logistic losses with the parameter $s$ being equal to $2$ (see~\cite{BartlettJordanMcauliffe06}).

%%%%%%%%%%%%%%%%%%%%%%%%%%%%%%%%%%%%%%%%%
\subsection{Data-driven procedure}
\label{subsec:dataDriven}
%%%%%%%%%%%%%%%%%%%%%%%%%%%%%%%%%%%%%%%%%

In this section we provide the steps of the construction of our empirical confidence set that is deduced from the empirical risk minimization. Before going into further details, let us first mention that our procedure is semi-supervised in the sense that it requires two datasets, one of which being unlabeled. Hence
 we introduce a first data set $\mathcal{D}_n = \{(X_i,Y_i), \; i = 1, \ldots, n\}$, 
which consists of independent copies of $(X,Y)$. 
We define the empirical $\phi$-risk associated to a score function~$f$ (which is the empirical counterpart of $R_\phi$ given in~\eqref{eq:phirisk}):
\begin{equation}
\label{eq:eqEmpPhiRisk}
\hat{R}_{\phi}(f) = \dfrac{1}{n} \sum_{i = 1}^n \sum_{k = 1}^K \phi(Z^i_k f_k(X_i)),
\end{equation}
where $Z^i_k = 2 \; \1_{\{Y_i = k\}}-1$ for all $k = 1, \ldots, K$.
We also define the empirical risk minimizer over $\mathcal{F}$, a convex set of score functions, as
\begin{equation*}
\hat{f} = \arg\min_{f \in \mathcal{F}} \hat{R}_{\phi}(f).
\end{equation*}
At this point, we have in hands the optimal score function $\hat{f}$ and need to build the corresponding confidence set with the right expected size.
However, let us before introduce an intermediate confidence set that would help comprehension since it mimics the oracle $\beta$-set $\Gamma^*$ with regard to its construction using $\hat{f}$ instead of $f^*$.
For this purpose, we define
\begin{equation*}
F_{\hat{f}_k}(t) = \mathbb{P}_{X}\left(\hat{f}_k(X) \leq t | \mathcal{D}_n\right),
\end{equation*}
for $t \in \mathbb{R}$, where $\mathbb{P}_{X}$ is the marginal distribution of $X$.
As for the c.d.f. of $p_k$ and $f_k$, we make the following assumption:
{\it 
\begin{itemize}
\item[{\bf (A2)}] The cumulative distribution functions $F_{\hat{f}_k}$ with $k=1,\ldots,K$ are continuous.
\end{itemize}
}
At this point, we are able to define an empirical approximation of the Oracle $\beta$-set for $\beta \in (0,K)$:
\begin{equation}
\label{eq:eqFirstEstim}
\widetilde{\Gamma}_{\beta}(X) = \left\{ k \in \{1,\ldots, K\} : \ \widetilde{G}(\hat{f}_{k}(X)) \leq \beta  \right\},
\end{equation} 
where for $t \in \mathbb{R}$
\begin{equation}
\label{eq:tildG}
\widetilde{G}(t) = \sum_{k=1}^K   \bar{F}_{\hat{f}_k}(t),
\end{equation}
with $ \bar{F}_{\hat{f}_k}= 1 -  F_{\hat{f}_k}$.
Since the function $\widetilde{G}$ depends on the unknown distribution of $X$, we consider a second but interestingly {\it unlabeled} dataset
$\mathcal{D}_N = \{X_i, i = 1, \ldots, N\}$, independent of $\mathcal{D}_n$ in order to compute the empirical versions of 
the $\bar{F}_{\hat{f}_k}$'s.
By now,  we can define the empirical $\beta$-set based on $\hat{f}$:
%%%%
%%%%
\begin{definition}
\label{def:empirConfSet}
Let $\hat{f}$ be the minimizer of the empirical $\phi$-risk given in~\eqref{eq:eqEmpPhiRisk} based on $\mathcal{D}_n$, and consider the unlabeled dataset $\mathcal{D}_N$. Let $\beta \in (0,K)$. 
The empirical $\beta$-set is given by
\begin{equation}
\label{eq:eqEmpConfSet}
\hat{\Gamma}_{\beta}(X) = \left\{ k \in \{1,\ldots, K\} : \ \hat{G}(\hat{f}_{k}(X)) \leq \beta  \right\},
\end{equation}
where 
\begin{equation*}
\hat{G}(.) = \frac{1}{N} \sum_{i = 1}^N \sum_{k = 1}^K \1_{\{\hat{f}_k(X_i) \geq .\}}.
\end{equation*}
\end{definition}
%%%%
%%%%
The most important remark about the construction of this data-driven confidence set is that it is made in a semi-supervised way. 
%Some remarks can be made from the definition of our data driven procedure.  
%First, it is important to note that the construction of the plug-in $\beta$-set is made in a semi-supervised way. 
Indeed, the estimation of the tail distribution functions of $\hat{f}_k$ requires only a set of {\it unlabeled} observations.
This is particularly attractive in applications where the number of label observations is small (because labelling examples is time consuming or may be costly) and where one has in hand several unlabeled features that can be used to make prediction more accurate.
As an important consequence is that the estimation of the tail distribution functions of $\hat{f}_k$ does not depend on the number of observations in each class label. That is to say, this unlabeled dataset can be unbalanced with respect to the classes, that can often occur in multiclass classification settings where the number of classes $K$ is quite large. 
Next section deals with the theoretical performance of the empirical $\beta$-sets.

%The following Proposition which is a consequence of Theorem~\ref{theo:theo1} ensures the convergence of our procedure
%in the following sense
%\begin{proposition}
%\label{prop:propCve}
%Let $\beta > G(1/2)$.
%We assume that $R_{\phi}(\hat{f}) - R_{\phi}(f_{\mathcal{F}}) \rightarrow 0$ then 
%
%$$\Delta R(\hat{\Gamma}_{\beta}) \overset{P}{\rightarrow} 0.$$
%
%\end{proposition} 

%%%%%%%%%%%%%%%%%%%%%%%%%%%%%%%%%%%%%%%%%
\subsection{Rates of convergence}
\label{subsec:rateOfConvergence}
%%%%%%%%%%%%%%%%%%%%%%%%%%%%%%%%%%%%%%%%%

In this Section, we provide rates of convergence for the empirical confidence sets defined in 
Section~\ref{subsec:dataDriven}.
First, we state some additional notation.
In the sequel, the symbols $\mathbf{P}$ and $\mathbf{E}$ stand for generic probability and expectation, respectively.
Moreover, given the empirical $\beta$-set $\hat{\Gamma}_{\beta}$ from Definition~\ref{def:empirConfSet}, we consider the risk 
$\mathbf{R}\left(\hat{\Gamma}_{\beta}\right) = \mathbf{P}\left(Y \notin \hat{\Gamma}_{\beta}(X)\right)$ and the 
information $\mathbf{I}\left(\hat{\Gamma}_{\beta}\right) = \mathbf{E}\left[  |  \hat{\Gamma}_{\beta}(X) | \right]$.
Our first result ensures that $\hat{\Gamma}_{\beta}$ is of level $\beta$ up to a term of order~$K/\sqrt{N}$. 
\begin{proposition}
\label{prop:confSetSize}
For each $\beta \in (0,K)$, the following equality holds
\begin{equation*}
\mathbf{I}\left(\hat{\Gamma}_{\beta}\right)   = \beta + O\left(\dfrac{K}{\sqrt{N}}\right).
\end{equation*}
\end{proposition}
In order to precise rates of convergence for the risk, we need to formulate several assumptions. First, we consider loss functions $\phi$ which have in common that the modulus of convexity of their underlying risk function $R_{\phi}$, defined by
\begin{equation}
\label{eq:deltaMod}
\delta(\varepsilon) = \inf\left\{\frac{R_{\phi}(f)+R_{\phi}(g)}{2} - R_{\phi}\left(\frac{f+g}{2}\right), \;\; \sum_{k=1}^K\mathbb{E}_{X}\left[(f_k-g_k)^2(X)\right] \geq \varepsilon^2\right\},
\end{equation}
satisfies $\delta(\varepsilon)\geq c_1\epsilon^2$ for some $c_1>0$ (we refer to~\cite{BartlettJordanMcauliffe06, MendelsonBartlett} for more details on modulus of convexity for classification risk). 
Moreover, we assume that $\phi$ is classification calibrated and $L$-Lipschitz for $L>0$.
On the other hand, for $\alpha > 0$, we assume a margin condition ${\rm M_{\alpha}^k}$ on each $p_k, k = 1,\ldots,K$.
\begin{equation*}
{\rm M_{\alpha}^k}:  \;\;\;\mathbb{P}_{X}\left(0<|p_k(X)-G^{-1}(\beta)|\leq t\right)\leq c_2t^{\alpha}, \;\;\; \text{for some constant} \;\; c_2>0 \;\; \text{and for all} \;\; t>0.
\end{equation*}
It is important to note that since we assume that the distribution functions of $p_k(X)$ are continuous for each $k$,
we have $\mathbb{P}_{X}\left(0<|p_k(X)-G^{-1}(\beta)|\leq t\right) \rightarrow 0$ with $t \rightarrow 0$. Therefore,
the margin condition only precise the rate of this decay to $0$.
Now, we provide the rate of convergence that we can expect for the empirical confidence sets defined by~\eqref{eq:eqEmpConfSet}.
\begin{theorem}
\label{theo:vitesse1}
Assume that $\|f\|_{\infty} \leq B$ for all $f \in \mathcal{F}$. 
Let $M_n = \mathcal{N}(1/n, L_{\infty}, \mathcal{F})$ be the 
cardinality of the set of closed balls with radius $1/n$ in $L_{\infty}$-norm needed to cover $\mathcal{F}$.
Under the assumptions of Theorem~\ref{theo:theo1}, with the modulos of convexity $\delta(\varepsilon)\geq c_1\epsilon^2$ with $c_1>0$, if $\phi$ is $L$-Lipschitz and if the margin assumptions ${\rm M_{\alpha}^k}$ are satisfied with $\alpha > 0$, then
\begin{equation*}
\mathbf{E}\left[|\Delta \mathcal R(\hat{\Gamma}_{\beta})|\right] \leq C(B,L,s, \alpha) K^{1-\alpha/(\alpha+s)}
\left\{\inf_{f\in \mathcal{F}} \Delta R_{\phi}(f) + \frac{K\log(M_n)}{n} \right\}^{\alpha/(\alpha+s)} + C^{'}\frac{K}{\sqrt{N}},
\end{equation*}
where $C^{'} >0$ is an absolute constant and $C(B,L,s,\alpha) > 0$ is a constant that depends only on $L,B,s$ and $\alpha$.
\end{theorem} 
From this theorem we obtain the following rate of convergence:
$K\left(\frac{\log(M_n)}{n}\right)^{\alpha/(\alpha+s)} + \frac{K}{\sqrt{N}}$.
This is the first bound for confidence sets in multiclass setting. The regularity parameter $\alpha$ plays a crucial role and governs the rate; larger values of $\alpha$ lead to faster rates. Moreover, this rate is linear in K, the number of classes, that seems to be the characteristic of multiclass classification as compared to binary classification.
Note that the exponent $\alpha/(\alpha+s)$ is not classical and is due to the estimation of quantiles.
The second part of the rates which is of order $K/ \sqrt{N}$ relies on the estimation of the function $\tilde{G}$ given in~\eqref{eq:tildG}. 
This part of the estimation is established under the mild Assumptions (A1) and (A2). For this term the proof of the linearity on $K$
is tricky and is obtained thanks to a new technical probabilistic results on sums of cumulative distribution functions (see Lemma~\ref{lemme:lm1}).
Let us conclude this paragraph by mentioning that Theorem~\ref{theo:vitesse1} applies for instance for $\mathcal{F}$ being the convex hull of a finite family of a score functions which is the scope of the next section.  
%%%%%%%%%%%%%%%%%%%%%%%%%%%%%%%%%%%%%%%%%
%%%%%%%%%%%%%%%%%%%%%%%%%%%%%%%%%%%%%%%%%
\section{Application to confidence sets aggregation}
\label{Sec:numRes}
%%%%%%%%%%%%%%%%%%%%%%%%%%%%%%%%%%%%%%%%%
%%%%%%%%%%%%%%%%%%%%%%%%%%%%%%%%%%%%%%%%%

This Section is devoted to an aggregation procedure 
which relies on the superlearning principle. The specifics of our superlearning procedure is given 
in Section~\ref{subsec:superlearning}. In Section~\ref{subsec:algoCve}, we show the consistency of our procedure and finally
we provide a numerical illustration in Section~\ref{subsec:realData} of our algorithm.

%%%%%%%%%%%%%%%%%%%%%%%%%%%%%%%%%%%%%%%%%%
\subsection{Description of the procedure: superlearning}
\label{subsec:superlearning}
%%%%%%%%%%%%%%%%%%%%%%%%%%%%%%%%%%%%%%%%%%

In this section, we describe the superlearning procedure for confidence sets.
The superlearning procedure is based on the $V$-fold cross-validation procedure (see~\cite{vanderLaanSL}).
Initially, this aggregation procedure has been introduced in the context of regression and binary classification.
Let $V \geq 2$ be an integer and let $(B_v)_{1 \leq v \leq V}$ be a regular partition of  $\{1, \ldots, n \}$, {\em i.e.}, a partition such that, for each $v= 1, \ldots, V$,
$\card (B_v) \in \{ \lfloor n/V \rfloor, \lfloor n/V \rfloor + 1 \}$, where we write $\lfloor x \rfloor$ for the floor of any real number $x$ (that is, $x-1 \leq \lfloor x \rfloor \leq x$). 
For each $v \in \{1, \ldots, V\}$, we denote by $D_n^{(v)}$ (respectively $D_n^{(-v)}$) the dataset $\{(X_i,Y_i), i \in B_v\}$ (respectively $\{(X_i,Y_i), i \not\in B_v\}$), 
and define the
corresponding empirical measures
\begin{eqnarray*}
P_{n}^{(v)} & = & \frac{1}{\card(B_v)} \sum_{i \in B_v} \text{Dirac}(X_i,Y_i), \quad\text{and}\\
P_{n}^{(-v)} & = & \frac{1}{n - \card(B_v)} \sum_{i \not\in B_v} \text{Dirac}(X_i,Y_i). \\
\end{eqnarray*} 
For a score algorithm $f$, that is to say a function which maps the empirical distribution to a score function.
We define the cross-validated risk of $f$ as
\begin{equation}
\label{eq:eqCrossVRisk}
\widehat{R}_{\phi}^n(f)  =  \frac{1}{V} \sum_{v = 1}^V R_{\phi}^{(P_n^{(v)})} \left( {f}(P_n^{(-v)}) \right),
\end{equation}
where
for each $v  \in \{1, \ldots, V\}$,  
$R_{\phi}^{(P_n^{(v)})}(f(P_{n}^{(-v)}))$ is the empirical estimator 
of  $R_{\phi} (f(P_n^{(-v)})$, based on $D_n^{(v)}$ and conditionally on $D_n^{(-v)}$:
\begin{equation*}
R_{\phi}^{(P_n^{(v)})} \left(f(P_n^{(-v)}) \right)   =  \dfrac{1}{\card(B_v)} \sum_{i 
\in I_v}  \sum_{k=1}^K \phi(Z_k^i f_k(P_n^{(-v)})(X_i)).
\end{equation*}
 Next, we consider $\mathcal{F} = \left(f^1, \ldots,f^M\right)$ a family of $M$ score algorithms.
 We define the cross-validated score by
 \begin{equation}
 \label{eq:eqCrossVScore}
 \widehat{f} \in \mathop{\argmin_{f \in \rm{conv}(\mathcal{F})}} \widehat{R}_{\phi}^n (f).
 \end{equation}
Finally, we consider the resulting cross-validated confidence set defined by
\begin{equation*}
\hat{\Gamma}^{{\rm CV}}_{\beta}(X) = \left\{ k \in \{1,\ldots, K\} : \ \hat{G}(\hat{f}_{k}(X)) \leq \beta  \right\},
\end{equation*} 
that we analyse in the next section.

 %%%%%%%%%%%%%%%%%%%%%%%%%%%%%%%%%%%%%%%%%%
\subsection{Consistency of the algorithm}
\label{subsec:algoCve}
%%%%%%%%%%%%%%%%%%%%%%%%%%%%%%%%%%%%%%%%%%
 
In this Section, we show some results that illustrate the consistency
of the superlearning  procedure described above.
To this end, we introduce the oracle counterpart of the cross-validated risk defined in~\eqref{eq:eqCrossVRisk}.
For a score $f$, we define
\begin{equation*}
\tilde{R}^n_{\phi}(f) = \frac{1}{V} \sum_{v= 1}^V R_{\phi} (f(P_n^{(-v)}),
\end{equation*}
that yields to the oracle counterpart of the cross-validated score defined in~\eqref{eq:eqCrossVScore} 
\begin{equation*}
\tilde{f} \in \argmin_{f \in \rm{conv}(\mathcal{F})} \tilde{R}^n_{\phi} (f).
\end{equation*}
Here again, we assume that the loss function $\phi$ satisfies the properties described in Section~\ref{subsec:rateOfConvergence} so that we can state the following result:
\begin{proposition}
\label{prp:crossVal}
We assume that for each $f \in {\rm conv}(\mathcal{F})$ and  $v \in \{1, \ldots,V\}$, $\|f(P_n^{(-v)})\|_{\infty} \leq B$.
Then the following holds 
\begin{equation*}
\mathbf{E}\left[\tilde{R}^n_{\phi}(\hat{f}) - \tilde{R}^n_{\phi}(\tilde{f})\right] \leq C(B,L) \dfrac{KM\log(n)}{\lfloor n/V \rfloor},
\end{equation*}
where $C(B,L)>0$ is a constant that depends only on $B$ and $L$.
\end{proposition}
As usual when one deals with cross-validated estimators,  the theorem compares $\hat{f}$ to the oracle counterpart 
$\tilde{f}$ in terms of the oracle cross-validated $\phi$-risk.
The theorem teaches us that, for sufficiently large $n$, $\hat{f}$    
perform as well as $\tilde{f}$.
However our main goal remains confidence sets. Therefore the next step consists in showing that the confidence set associated to the cross-validated score $\hat f$ has good properties in terms of the classification risk. Let  $\Gamma_{f,\delta}$ be the confidence set that results from a choice of $\beta \in (0,K)$ and a score function $f \in {\rm conv}(\mathcal{F})$.
We introduce the following excess risks
\begin{eqnarray*}
\Delta \tilde{\mathcal{R}}^n(\Gamma_{f,\delta}) & = & \frac{1}{V} \sum_{v = 1}^V \mathcal{R}\left(\Gamma_{f\left(P_n^{(-v)}\right),\delta}\right) - \mathcal{R}^*,\\
\Delta \tilde{R}^n_{\phi}(f) & = & \frac{1}{V} \sum_{v = 1}^V R_{\phi}(f(P_n^{(-v)})) - R_{\phi}(f^*). 
\end{eqnarray*}
These quantities can be view as cross-validated counterparts of the excess risks introduced in Section~\ref{subsec:confSetCalib}. Hereafter,
we provide a result which can be view as the cross-validation counterpart of Theorem~\ref{theo:vitesse1}. 
It is interpreted in a same way.
\begin{proposition}
\label{prp:OracleCrossVal}
We assume that for each $v \in \{1, \ldots, V\}$, $t \mapsto \mathbb{P}_{X}\left(\hat{f}\left(P_n^{(-v)}\right) \leq t|\mathcal{D}_n\right)$ is continuous.
Under the assumptions of Proposition~\ref{prp:crossVal} and  if the margin assumptions ${\rm M_{\alpha}^k}$ are satisfied with $\alpha > 0$,
\begin{equation*}
\mathbf{E}\left[|\Delta \tilde{\mathcal{R}}^n(\hat{\Gamma}_{\beta}^{{\rm CV}})|\right] \leq C(B,L,\alpha,s) K^{1-\alpha/(\alpha+s)}
\left\{\mathbf{E}\left[\Delta \tilde{R}^n_{\phi}(\tilde{f})\right] + \frac{KM\log(n)}{\lfloor n/V \rfloor} \right\}^{\alpha/(\alpha+s)} + C^{'}\frac{K}{\sqrt{N}}.
\end{equation*}   
\end{proposition}
The proof of Proposition~\ref{prp:OracleCrossVal} relies on Proposition~\ref{prp:crossVal} and similar arguments as in Theorem~\ref{theo:vitesse1}. 
%%%%%%%%%%%%%%%%%%%%%%%%%%%%%%%%%%%%%%%%%%
\subsection{Application to real datasets}
\label{subsec:realData}
%%%%%%%%%%%%%%%%%%%%%%%%%%%%%%%%%%%%%%%%%%

\begin{table}
\begin{center}
\footnotesize{
\begin{tabular}{l | c || ccccc} 
\multicolumn{2}{c}{} & \multicolumn{5}{c}{{$\texttt{Forest}$ ($K=4$)}}\\
\hline 
\multicolumn{2}{c}{} &  \multicolumn{5}{c}{$\beta$-set} \\
\hline\noalign{\smallskip}
 $\beta$  &  & \;\;   \texttt{rforest} \;\;\; &  \texttt{softmax reg} \;\;\;& \texttt{svm} \;\;\; &  \texttt{kknn} \;\;\; &   \texttt{CV} \\
\noalign{\smallskip}
\hline
\noalign{\smallskip}
 2  & $\mathbf{R}$ & 0.02 (0.02) & 0.06 (0.02) & 0.02 (0.01) & 0.05 (0.03) & 0.02 (0.01) \\
   & $\mathbf{I}$ &  2.00 (0.09) & 2.00 (0.08) & 2.00 (0.09) & 2.00 (0.08) & 2.00 (0.08)  \\ \hline
\end{tabular}
}

\vspace*{0.25cm}

\footnotesize{
\begin{tabular}{l | c || ccccc} 
\multicolumn{2}{c}{} & \multicolumn{5}{c}{{$\texttt{Plant}$ ($K=100$)}}\\
\hline 
\multicolumn{2}{c}{} &  \multicolumn{5}{c}{$\beta$-set} \\
\hline\noalign{\smallskip}
 $\beta$  &  & \;\;   \texttt{rforest} \;\;\; &  \texttt{softmax reg} \;\;\;& \texttt{svm} \;\;\; &  \texttt{kknn} \;\;\; &   \texttt{CV} \\
\noalign{\smallskip}
\hline
\noalign{\smallskip}
 2  & $\mathbf{R}$ &  0.18 (0.03) &   0.77 (0.02) & 0.32 (0.04) & 0.20 (0.03) & 0.17 (0.03)\\
   & $\mathbf{I}$ &  2.00 (0.09) & 2.02 (0.18) & 1.99 (0.10) & 2.00 (0.08) & 2.00 (0.08) \\
 10  & $\mathbf{R}$ & 0.02 (0.01) & 0.42 (0.04) & 0.03 (0.02) & 0.08 (0.03) & 0.02 (0.01)  \\
   & $\mathbf{I}$ &   9.95 (0.38) & 10.06 (0.58) & 9.98 (0.22) & 9.98 (0.23) & 9.96 (0.37)\\
\hline
\end{tabular}
}
\caption{\label{table:tableRealData}
For each of the $B = 100$ repetitions and for each dataset, we derive the estimated 
risks~$\mathbf{R}$ and information~$\mathbf{I}$ of the different $\beta$-sets w.r.t. $\beta$.
We compute the means and standard deviations (between parentheses) over the $B = 100$ repetitions. 
For each $\beta$, the $\beta$-sets are based on, from left to right, 
\texttt{rforest}, \texttt{softmax reg} and \texttt{svm}, \texttt{kknn} and \texttt{CV} 
which are respectively the random forest, the softmax regression, support vector machines, $k$ nearest neighbors and the superlearning procedure.
Top: the dataset is the $\texttt{Forest}$ --  the dataset is the $\texttt{Plant}$.}
\end{center}
\end{table}

In this section, we provide an application of our aggregation procedure described in Section~\ref{subsec:superlearning}.
For the numerical experiment we focus on the boosting loss and consider the library of algorithms constituted 
by the random forest, the softmax regression,  the support vector machines and  the $k$ nearest neighbors (with $k = 11$) procedures.
To be more specifics, we respectively exploit the \texttt{R} packages \texttt{randomForest}, \texttt{polspline}, \texttt{e1071} and \texttt{kknn}.
All the \texttt{R} functions are used with standard tuning parameters. Finally the parameter $V$ of the aggregation procedure is fixed to $5$.
 
We evaluate the performance of the procedure on two real datasets: the {\it Forest type mapping} dataset and the {\it one-hundred plant species leaves} dataset coming from the UCI database.
In the sequel we refer to these two datasets as $\texttt{Forest}$ and $\texttt{Plant}$ respectively.
The $\texttt{Forest}$ dataset consists of $K=4$ classes and $523$ labeled observations (we gather the train and test sets) with $27$ features.
Here the classes are unbalanced.
In the $\texttt{Plant}$ dataset, there are $K=100$ classes and $1600$ labeled observations.
This dataset is balanced so that each class consists of $16$ observations.
The original dataset contains $3$ covariates (each covariate consists of $64$ features).
In order to make the problem more challenging, we drop $2$ covariates.

To get an indication of the statistical significance, it makes sense to compare our aggregated confidence set (referred as \texttt{CV}) to the confidence sets that result from each component of the library. 
Roughly speaking, we evaluate risks (and informations) of these confidence sets on each dataset.
To do so, we use the {\it cross validation} principle. 
In particular, we run $B=100$ times the procedure where we split the data each time in three: 
a sample of size $n$ to build the scores $\hat{f}$;
a sample of size $N$ to estimate the function $G$ and to get the confidence sets;
and a sample of size $M$ to evaluate the risk and the information.
For both datasets, we make sure that in the sample of size $n$, there is the same number of observations in each class.
As a benchmark, we notify that the misclassification risks of the best classifier from the library  in the
$\texttt{Forest}$ dataset is evaluated at $0.15$ , whereas in the $\texttt{Plant}$ dataset, it is evaluated at $0.40$.
As planned, the performance of the classical methods are quite bad in this last dataset.
 
We set the sizes of the samples as $n=200$, $N=100$ and $M=223$ for the $\texttt{Forest}$ dataset, and $n=1000$, $N=200$ and $M=400$ for the $\texttt{Plant}$ one.
The results are reported in Table~\ref{table:tableRealData}, and confirm our expectations.
In particular, our main observation is that the aggregated confidence set (\texttt{CV}) outperforms all components of the library in the sense that
it is at least as good as the best component in all of the experiments.
Second, let us state some remarks that hold for all of the confidence sets and in particular our aggregated confidence set. First, we note that the information $\mathbf{I}(\Gamma)$ has the good level $\beta$ which is supported by our theory. Moreover, we see that the risk gets drastically better with moderate $\beta$ as compared to the {\it best} misclassification risk. 
For instance, for the $\texttt{Plant}$,  the error rate of the confidence set with $\beta =2$ based on random forests is $0.18$ whereas the misclassification error rate
of the best component is $0.40$.

%%%%%%%%%%%%%%%%%%%%%%%%%%%%%%%%%%%%%%%%%
%%%%%%%%%%%%%%%%%%%%%%%%%%%%%%%%%%%%%%%%%
\section{Conclusion}
\label{Sec:Concl}
%%%%%%%%%%%%%%%%%%%%%%%%%%%%%%%%%%%%%%%%%
%%%%%%%%%%%%%%%%%%%%%%%%%%%%%%%%%%%%%%%%%

In multiclass classification setting, the present paper propose a new procedure that assigns a set of labels  instead of a single label to each instance. 
This set has a controlled expected size (or information) and its construction relies on cumulative distribution functions and on an empirical risk minimization procedure.
Theoretical guarantees, especially rates of convergence are also provided and rely on the regularity of these cumulative distribution functions.
The obtained rates of convergence highlight a linear dependence w.r.t  the number of classes $K$. 
The procedure described in Section~\ref{Sec:EmpRisk} is defined as a two steps algorithm whose second step consists in the estimation of the function $G$ (which is a sum of tail distribution functions).
Interestingly, this step does not require a set of labeled data and neither to explore the whole classes, that is suitable for
semi-supervised learning.
Moreover, we apply our methodology to derive an aggregation algorithm which is based on the $V$-fold cross-validation principle.  
Future works will focus on the optimality in the minimax sense with respect to the classification error. 
In particular, we will investigate whether the rates of convergence are optimal in terms of their dependence on $K$.
We believe that this dependence (linearity on $K$) is the correct one.
However we will instigate whether this dependence can be reduced under some sparsity assumption.

%%%%%%%%%%%%%%%%%%%%%%%%%%%%%%%%%%%%%%%%%
%%%%%%%%%%%%%%%%%%%%%%%%%%%%%%%%%%%%%%%%%
\section{Appendix}
\label{Sec:prof}
%%%%%%%%%%%%%%%%%%%%%%%%%%%%%%%%%%%%%%%%%
%%%%%%%%%%%%%%%%%%%%%%%%%%%%%%%%%%%%%%%%%

This section gathers the proofs of our results. 
Let us first add a notation that will be used throughout the Appendix:
for any random variable (or vector) $Z$, we denote by $\mathbb{P}_Z$ the probability w.r.t. $Z$ and by 
$\mathbb{E}_Z$, the corresponding expectation.

%%%%%%%%%%%%%%%%%%%%%%%%%%%%%%%%%%%%%%%%%%
\subsection{Technical Lemmas}
\label{subsec:lemma}
%%%%%%%%%%%%%%%%%%%%%%%%%%%%%%%%%%%%%%%%%%

We first lay out key lemmata, which are crucial to establish the main theory.
We consider $K \geq 2$ be an integer, and $Z_1, \ldots, Z_K$, $K$ random variables.
Moreover we define function $H$ by:
\begin{equation*}
H(t) = \dfrac{1}{K} \sum_{k = 1}^K F_k(t), \;\; \forall t \in [0,1],
\end{equation*} 
where for all $k = 1, \ldots , K$, $F_k$ is the cumulative distribution function of $Z_k$.
%For every $j$, we denote by $\hat{F}_j$ the empirical cumulative distribution function of $Z_j$
%based on some sample $\mathcal{D}_n$ of size $N$.
%Note that the estimators $\hat{F}_j$ can rely on the same sample.
Finally, let us define the generalized inverse $H^{-1}$ of $H$:
\begin{equation*}
H^{-1}(p) = \inf\{t: H(t) \geq p \}, \;\; \forall p \in (0,1).
\end{equation*}

\begin{lm}
\label{lemme:lm1}
Let $\varepsilon$ distributed from a uniform distribution on $\{1, \ldots, K\}$
and independent of $Z_k$, $k = 1, \ldots, K$. Let $U$ distributed from a uniform distribution 
on $[0,1]$.
We consider 
\begin{equation*}
Z = \sum_{k = 1}^K Z_k \1_{\{\varepsilon = k\}}.
\end{equation*}
If $H$ is continuous then
\begin{equation*}
H(Z)  \overset{\mathcal{L}}{=} U \;\; {\rm and} \;\; H^{-1}(U)  \overset{\mathcal{L}}{=} Z
\end{equation*}
\end{lm}
\begin{proof}
First we note that for every $t \in [0,1]$,
$\mathbb{P}\left(H(Z) \leq t\right) = \mathbb{P}\left(Z \leq H^{-1}(t)\right)$.
Moreover, we have
\begin{eqnarray*}
\mathbb{P}\left(H( Z)\leq t\right) & = & \sum_{k = 1}^K \mathbb{P}(Z \leq H^{-1}(t), \varepsilon = k) \\
                                & = & \dfrac{1}{K} \sum_{k = 1}^K \mathbb{P}(Z_k \leq H^{-1}(t)) \;\; {\rm with} \;\; 
                                \varepsilon \;\; {\rm independent \;\; of} \;\; X\\
                                   & = & H(H^{-1}(t)) \\
                                   & = & t \;\; {\rm with} \;\; H \;\; {\rm continuous}.
\end{eqnarray*}
To conclude the proof, we observe that
\begin{eqnarray*}
\mathbb{P}\left(H^{-1}(U) \leq t\right) & = & \mathbb{P}\left(U \leq H(t)\right) \\
                                        & = & \dfrac{1}{K} \sum_{k = 1}^K F_k(t) \\
                                        & = & \sum_{k = 1}^K \mathbb{P}\left(Z_k \leq t, \varepsilon  = k\right)\\
                                        & = & \mathbb{P}(Z\leq t).
\end{eqnarray*}
\end{proof}

\begin{lm}
\label{lem:lemFondamental}
There exists an absolute constant $C' > 0$ such that
\begin{equation*}
\sum_{k = 1}^K \mathbf{P}\left(|\hat{G}(\hat{f}_k(X) - \widetilde{G}(\hat{f}_k(X))| \geq |\widetilde{G}(\hat{f}_k(X)) -\beta|
\right) \leq \dfrac{C'K}{\sqrt{N}}.
\end{equation*}
\end{lm}
\begin{proof}
We define, for $\gamma > 0$ and $k \in \{1, \ldots K\}$
\begin{eqnarray*}
{A}^k_0 & = & \left\lbrace\left|\widetilde{G}(\hat{f}_k(X)) -\beta\right| \leq \gamma\right\rbrace\\
{A}^k_j & = & \left\{2^{j-1}\gamma  < |\widetilde{G}(\hat{f}_k(X)) -\beta| \leq 2^{j} \gamma \right\} , \;\; j \geq 1.
\end{eqnarray*}
Since, for every $k$, the events $({A}^k_j)_{j\geq 0}$ are mutually exclusive, we deduce 
\begin{multline}
\label{eq:eqDec}
\sum_{k = 1}^K \mathbf{P}\left(|\hat{G}(\hat{f}_k(X) - \widetilde{G}(\hat{f}_k(X))| \geq |\widetilde{G}(\hat{f}_k(X)) -\beta|
\right) = \\
\sum_{k = 1}^K \sum_{j\geq 0} \mathbf{P}\left(|\hat{G}(\hat{f}_k(X) - \widetilde{G}(\hat{f}_k(X))| \geq |
\widetilde{G} (\hat{f}_k(X)) -\beta| , A_j^k\right).
\end{multline}
Now, we consider a random variable $\varepsilon$ uniformly distributed on $\{1, \ldots, K\}$ and independent of $\mathcal{D}_n$ and $X$.
Conditional on $\mathcal{D}_n$ and under Assumption $(A2)$, we apply Lemma~\ref{lemme:lm1} with
$Z_k = \hat{f}_k(X)$, $Z = \sum_{k = 1}^K Z_k \1_{\{\varepsilon = k\}}$ and then obtain that
$\widetilde{G}(Z)$ is uniformly distributed on $[0,K]$. Therefore, for all $j \geq 0$ and $\gamma > 0$, we deduce
\begin{eqnarray*}
 \dfrac{1}{K} \sum_{k=1}^K \mathbb{P}_X\left(|\widetilde{G}(\hat{f}_k(X)) - \beta | \leq 2^j \gamma | \mathcal{D}_n\right) & = & 
\mathbb{P}_X\left(|\widetilde{G}(Z) - \beta | \leq 2^j \gamma | \mathcal{D}_n\right) \\
& \leq & \dfrac{2^{j+1} \gamma}{K}.
\end{eqnarray*}
Hence, for all $j \geq 0$, we obtain
\begin{equation}
\label{eq:eqEns}
\sum_{k = 1}^K \mathbf{P}(A_j^k) \leq 2^{j+1} \gamma.
\end{equation}
Next, we observe that for all $j \geq 1$
\begin{multline}
\label{eq:eqAux}
\sum_{k = 1}^K \mathbf{P}\left(|\hat{G}(\hat{f}_k(X) - \widetilde{G}(\hat{f}_k(X))| \geq |
\widetilde{G} (\hat{f}_k(X)) -\beta| , A_j^k\right)  \leq \\
\sum_{k = 1}^K \mathbb{E}_{(\mathcal{D}_n, X)}\left[\mathbb{P}_{\mathcal{D}_N}\left(|\hat{G}(\hat{f}_k(X)) - \widetilde{G}(\hat{f}_k(X))| 
\geq 2^{j-1}\gamma |\mathcal{D}_n, X\right)\1_{A_j^k} \right].
\end{multline} 
Now, since conditional on $(\mathcal{D}_n, X)$,  $\hat{G}(\hat{f}_k(X))$ is an empirical mean of i.i.d random variables 
of common mean $\widetilde{G}(\hat{f}_k(X)) \in [0,K]$, we deduce from Hoeffding's inequality that
\begin{equation*}
\mathbb{P}_{\mathcal{D}_N}\left(|\hat{G}(\hat{f}_k(X)) - \widetilde{G}(\hat{f}_k(X))| 
\geq 2^{j-1}\gamma |\mathcal{D}_n, X\right) \leq 2\exp\left(-\dfrac{N\gamma^2 2^{2j-1}}{K^2}\right).
\end{equation*}
Therefore, from Inequalities~\eqref{eq:eqDec},~\eqref{eq:eqEns} and~\eqref{eq:eqAux}, we get
\begin{multline*}
\sum_{k = 1}^K \mathbf{P}\left(|\hat{G}(\hat{f}_k(X) - \widetilde{G}(\hat{f}_k(X))|  \geq  |\widetilde{G}(\hat{f}_k(X)) -\beta|\right)\\ 
\leq  \sum_{k = 1}^K \mathbf{P}\left(A_0^k\right) + \sum_{j \geq  1} 
2\exp\left(-\dfrac{N\gamma^2 2^{2j-1}}{K^2}\right)\left(\sum_{k=1}^K\mathbf{P}\left(A_j^k\right)\right) \\ 
 \leq  2\gamma + \gamma \sum_{j \geq 1} 2^{j+2} \exp\left(-\dfrac{N\gamma^2 2^{2j-1}}{K^2}\right).\\
\end{multline*}
Finally, choosing $\gamma = \dfrac{K}{\sqrt{N}}$ in the above inequality, we finish the proof of the lemma.    
\end{proof}

%%%%%%%%%%%%%%%%%%%%%%%%%%%%%%%%%%%%%%%%%
\subsection{Proof of Proposition~\ref{prop:prop1}}
\label{subsec:proofpropOracle}
%%%%%%%%%%%%%%%%%%%%%%%%%%%%%%%%%%%%%%%%%

Let $\beta > 0$ and  $\Gamma$ be a confidence set such that $\mathcal{I}(\Gamma) = \beta$.
First, we note that the following decomposition holds
\begin{multline*}
\mathcal{R}(\Gamma) - \mathcal{R}(\Gamma_{\beta}^{*})  =  \sum_{j = 1}^K \sum_{k = 1}^K\mathbb{E}\left[ \sum_{l = 1}^{K} (\1_{\{Y \notin \Gamma(X)\}}- 
                                                                       \1_{\{Y \notin \Gamma_{\beta}^*(X)\}})\1_{\{Y = l\}} \1_{\{|\Gamma(X)| = k\}} 
                                                                       \1_{\{| \Gamma_{\beta}^*(X)| = j\}}\right]\\
                                                       =  \sum_{j = 1}^K \sum_{k = 1}^K \mathbb{E}\left[\sum_{l=1}^{K} (\1_{\{l \notin \Gamma(X)\}}- \1_{\{l \notin 
                                                      \Gamma_{\beta}^*(X)\}}) p_{l}(X) \1_{\{|\Gamma(X)| = k\}} \1_{\{|\Gamma_{\beta}^*(X)| = j\}} \right]\\
                                                                           =  \sum_{j = 1}^K \sum_{k = 1}^K \mathbb{E}\left[\sum_{l=1}^K (\1_{\{l \in \Gamma_{\beta}^*(X) 
                                                                          \setminus \Gamma(X)\}}                                                                                                                                                                                                                                                                                                                                                                                                                                     
                                                                          -  \1_{\{l \in \Gamma(X) \setminus \Gamma^*_{\beta}(X)\}}) p_{l}(X) \1_{\{|\Gamma(X)| = k,
                                                                          | \Gamma_{\beta}^*(X)| = j\}} \right],                                                                      
\end{multline*}
where we conditioned by $X$ to get the second equality.
From the last decomposition and with 
\begin{eqnarray*}
\mathbb{E}\left[|\Gamma(X)|\right] & = & \mathbb{E}\left[|\Gamma_{\beta}^*(X)|\right] \\
                                   & = &   \mathbb{E}\left[ \sum_{j = 1}^K \sum_{k = 1}^K k \1_{\{|\Gamma(X)| = k\}} \1_{\{|\Gamma_{\beta}^*(X)| = j\}} \right] \\
                                   & = & \mathbb{E}\left[ \sum_{j = 1}^K \sum_{k = 1}^K j \1_{\{|\Gamma(X)| = k\}} \1_{\{|\Gamma_{\beta}^*(X)| = j\}}\right] ,
\end{eqnarray*}
we can express the excess risk as the sum of two terms:

\begin{multline}
\label{eq:eqprop1}
\mathcal{R}(\Gamma) - \mathcal{R}(\Gamma_{\beta}^{*}) = \sum_{j = 1}^K \sum_{k = 1}^K \mathbb{E}\left[\left(\sum_{l=1}^K \1_{\{l \in \Gamma_{\beta}^*(X) \setminus 
\Gamma(X)\}}                                                                                                                                                                                                                                                                                                                                                                                             
                                                                                                        p_{l}(X)- j G^{-1}(\beta)\right) \1_{\{|\Gamma(X)| = k\}}
                                                                                                         \1_{\{|\Gamma_{\beta}^*(X)| = j\}} \right] \\
  +        \sum_{j = 1}^K \sum_{k = 1}^K \mathbb{E}\left[\left(k G^{-1}(\beta) - \sum_{l=1}^K \1_{\{l \in \Gamma(X) \setminus \Gamma_{\beta}^*(X)\}}                                                                                                                                                                                                                                                                                                                                                                                             
                                                                                                        p_{l}(X)\right)  \1_{\{|\Gamma(X)| = k\}}
                                                                                                         \1_{\{|\Gamma_{\beta}^*(X)| = j\}} \right].                                                                                                  
\end{multline}
Now, for $j,k \in \{1, \ldots, K\}$ on the event $\{|\Gamma(X)| = k, |\Gamma_{\beta}^*(X)| = j\}$, we have
\begin{equation*}
k  =  \sum_{l = 1}^K  \1_{\{l \in \Gamma(X) \setminus \Gamma_{\beta}^*(X)\}} + \sum_{l = 1}^K \1_{\{l \in \Gamma(X) \cap \Gamma_{\beta}^*(X)\}},
\end{equation*}
and
\begin{equation*}
j  =   \sum_{l = 1}^K  \1_{\{l \in \Gamma_{\beta}^*(X) \setminus \Gamma(X)\}} + \sum_{l = 1}^K \1_{\{l \in \Gamma(X) \cap \Gamma_{\beta}^*(X)\}}.                                                                                                                                                                                                                                                                                                                                                                                               
\end{equation*}
Therefore, since
\begin{equation*} 
l \in \Gamma_{\beta}^* \Leftrightarrow p_{l}(X) \geq G^{-1}(\beta), 
\end{equation*}
Equality~\eqref{eq:eqprop1} yields the result.

%%%%%%%%%%%%%%%%%%%%%%%%
%%%%%%%%%%%%%%%%%%%%%%%%

\subsection{Proof of Theorem~\ref{theo:theo1}}
\label{subsec:proofTheo1}
First we recall that $\hat{f}_n = (\hat{f}_{n,1}, \ldots, \hat{f}_{n,K})$ is a sequence of score functions and $\delta_n\in \mathbb{R}$ is such that $G_{\hat{f}_n}(-\delta_n) =\beta$.
We suppress the dependence on $n$ to simplify notation and write $\hat{f}= (\hat{f}_{1}, \ldots, \hat{f}_{K})$ and $\delta$ for $\hat{f}$ and $\delta_n$ respectively.
Moreover, since there is no doubt, we also suppress everywhere the dependence on $X$.
We also define the events
\begin{equation}
\label{eq:proofSetBk}
B_{k} = \{\hat{f}_{k} \in (-\delta,-\delta^{*}) \;\; {\rm or} \;\; \hat{f}_{k} \in 
(-\delta^{*},-\delta)\},
\end{equation}
for $k = 1, \ldots, K$.
We aim at controlling the excess risk $\Delta \mathcal{R}(\Gamma_{\hat{f},\delta})$.
Since the risk $\mathcal{R}$ is decomposable, it is convenient to introduce ``marginal excess risks'':
\begin{equation*}
\Delta \mathcal{R}^{k}(\Gamma_{f,\delta}) = \1_{\{k\in (\Gamma_{\hat{f}, \delta} \, \Delta  \, \Gamma_\beta^{*})\}} |p_k - G^{-1}(\beta)|.
\end{equation*}
%First, we observe that
%\begin{eqnarray}
%\label{eq:eqDecompRiskTheo1}
%\Delta R(\Gamma_{f,\delta})^s & \leq & \mathbb{E}\left[ \sum_{k=1}^K \1_{\{k \in \Gamma_{f,\delta} \Delta 
%\Gamma_{\beta}^{*} \}} |p_k-G^{-1}(\beta)|\right]^s \nonumber \\
% & \leq & \mathbb{E}\left[\sum_{k=1}^K \frac{1}{K}\1_{\{k \in \Gamma_{f,\delta} \Delta 
%\Gamma_{\beta}^{*} \}}K |p_k-G^{-1}(\beta)| \right]^s \nonumber \\
%& \leq & \dfrac{1}{K} \sum_{k = 1}^K \mathbb{E}\left[\1_{\{k \in \Gamma_{f,\delta} \Delta 
%\Gamma_{\beta}^{*} \}} K^s |p_k-G^{-1}(\beta)|^s \right] \;\;({\rm by \;\; convexity}).
%\end{eqnarray}
%
%
Recall also that by convexity of the loss function $\phi$, we have that for all $x,y\in \mathbb{R}$,
\begin{equation}
\label{eq:eqUtile}
\phi(y)-\phi(x) \geq \phi'(x)(y-x).
\end{equation}
Assume that $\hat{f}_{k} \leq -\delta$ and $\hat{f}_{k} \leq -\delta^{*}\leq f_{k}^{*}$, which translates as $p_k-G^{-1}(\beta) \geq 0$, we get thanks to~\eqref{eq:eqUtile}
\begin{equation*}
p_k(\phi(\hat{f}_{k})-\phi(-\delta)) - (1-p_k)(\phi(-\hat{f}_{k})-\phi(\delta))\geq (\phi'(\delta^*)-\phi'(-\delta^*))(p_k-
G^{-1}(\beta))(\hat{f}_{k} + \delta^{*}) \geq 0 .
\end{equation*}
Similarly, 
if $\hat{f}_{k} \geq -\delta$ and $\hat{f}_{k} \geq -\delta^{*}\geq f_{k}^{*}$, that is, $p_k-G^{-1}(\beta) \leq 0$ we have
\begin{equation*}
p_k(\phi(\hat{f}_{k})-\phi(-\delta)) - (1-p_k)(\phi(-\hat{f}_{k})-\phi(\delta)) \geq 0. 
\end{equation*}
Note that in the following two cases
\begin{eqnarray*}
 \bullet \quad  \hat{f}_{k} \leq \ - \delta & {\rm and} & f_{k}^{*} \leq -\delta^{*};\\ 
\bullet \quad \hat{f}_{k} \geq \ - \delta & {\rm and} & f_{k}^{*} \geq -\delta^{*},
\end{eqnarray*}
we have $\Delta \mathcal{R}^{k}(\Gamma_{\hat{f},\delta}) = 0$.
Therefore, from the above inequalities, on $B_{k}^{c}$ and by assumption~\eqref{eq:Zhansgg} we get 
\begin{equation}
\label{eq:defout}
\1_{B_k^c}\1_{\{k\in (\Gamma_{\hat{f}, \delta} \, \Delta  \, \Gamma_\beta^{*})\}} |p_k - G^{-1}(\beta)|^s \leq C(\Delta R_{\phi}^k(\hat{f})).
\end{equation}
Therefore, since $s \geq 1$, we have
\begin{eqnarray*}
\left(\mathbb{E}\left[\sum_{k = 1}^K \1_{B_k^c \cap \{k\in (\Gamma_{\hat{f}, \delta} \, \Delta  \, \Gamma_\beta^{*})\}} |p_k - G^{-1}(\beta)|\right]\right)^s &\leq & \dfrac{1}{K} \sum_{k = 1}^K \mathbb{E}\left[\1_{B_k^c \cap \{k \in \Gamma_{\hat{f},\delta} \Delta 
\Gamma_{\beta}^{*} \}} K^s |p_k-G^{-1}(\beta)|^s \right] \\
& \leq &
 C K^{s-1} \Delta R_{\phi}(\hat{f}).  
\end{eqnarray*}
Moreover 
\begin{eqnarray*}
\mathbb{E} \left[\sum_{k = 1}^K \1_{B_k}\1_{\{k\in (\Gamma_{\hat{f}, \delta} \, \Delta  \, \Gamma_\beta^{*})\}} |p_k - G^{-1}(\beta)| \right] & \leq &
 \sum_{k = 1}^{K} \mathbb{P}(B_k)  \\
 & \leq & \mathbb{E}\left[|G_{\hat{f}}(-\delta) -G_{\hat{f}}(-\delta^{*})|\right] \\
 & \leq &  \mathbb{E}\left[|G_{f^{*}}(-\delta^{*}) -G_{\hat{f}}(-\delta^{*})|\right].
\end{eqnarray*}
Finally, we get the following bound
\begin{equation*}
\Delta \mathcal{R} (\Gamma_{\hat{f},\delta}) \leq K^{\frac{s-1}{s}}  \Delta R_{\phi}(\hat{f})^{1/s} + \mathbb{E}\left[|G_{f^{*}}(-\delta^{*}) -G_{\hat{f}}(-\delta^{*})|\right].
\end{equation*}
Now we observe that
\begin{equation}
\label{eq:eqDecompRisk1}
|\1_{\{\hat{f}_{k} \geq -\delta^{*}\}} - \1_{\{f_k^{*} \geq -\delta^{*}\}} | \leq \1_{\{|p_k - G^{-1}(\beta)|^s \leq 
C \Delta R_{\phi}^{k}(\hat{f})\}}.
\end{equation}
Therefore, for each $\gamma > 0$, we have
\begin{eqnarray*}
\mathbb{E}_{X}\left[|G_{f^{*}}(-\delta^{*}) -G_{\hat{f}}(-\delta^{*})|\right]  & \leq & 
\sum_{k = 1}^K \mathbb{P}_{X}\left(|p_k(X) - G^{-1}(\beta)|^s \leq C \Delta R_{\phi}^{k}(\hat{f})\right)\\
& \leq & \sum_{k=1}^K \mathbb{P}_{X}\left(|p_k(X) - G^{-1}(\beta)| \leq \gamma^{1/s} \right)
+  \mathbb{P}_{X}\left(\gamma \leq C \Delta R_{\phi}^{k}(\hat{f})\right).\\
\end{eqnarray*}
Now using Markov Inequality, we have that
\begin{eqnarray*}
\sum_{k=1}^K \mathbb{P}_{X}\left(\gamma \leq  C \Delta R_{\phi}^{k}(\hat{f})\right) & \leq &
\dfrac{C}{\gamma} \sum_{k =1}^K \mathbb{E}_{X} \left[ \Delta R_{\phi}^{k}(\hat{f})\right]\\ 
& \leq & \dfrac{C \Delta R_{\phi}(\hat{f})}{\gamma}.\\
\end{eqnarray*}
The above inequality yields
\begin{equation}
\label{eq:eqMajoration}
\mathbb{E}_{X}\left[|G_{f^{*}}(-\delta^{*}) -G_{\hat{f}}(-\delta^{*})|\right] \leq
 \dfrac{C \Delta R_{\phi}(\hat{f})}{\gamma} + \sum_{k=1}^K \mathbb{P}_{X}\left(|p_k(X) - G^{-1}(\beta)| \leq 
 \gamma^{1/s} \right).
\end{equation}
Hence, with Equation~\eqref{eq:eqDecompRisk1}, we get
\begin{equation*}
\Delta \mathcal{R} (\Gamma_{\hat{f},\delta}) \leq K^{\frac{s-1}{s}}  \Delta R_{\phi}(\hat{f})^{1/s} + \dfrac{C \Delta R_{\phi}(\hat{f})}{\gamma} + \sum_{k=1}^K \mathbb{P}_{X}\left(|p_k(X) - 
G^{-1}(\beta)| \leq \gamma^{1/s} \right).
\end{equation*}
The term $\sum_{k=1}^K \mathbb{P}_{X}\left(|p_k(X) - G^{-1}(\beta)| \leq \gamma^{1/s} \right) \rightarrow 0$ when $\gamma \rightarrow 0$, given that the distribution function of the $p_k's$ are continuous. Then using the convergence in distribution of $\Delta R_{\phi}(\hat{f})$ to zero, the last inequality ensures the desired result. 

%%%%%%%%%%%%%%%%%%%%%%%%%%%%%%%%%%%%%%%%
\subsection{Proof of Proposition~\ref{prop:confSetSize}}
\label{subsec:proofPropSize}
%%%%%%%%%%%%%%%%%%%%%%%%%%%%%%%%%%%%%%%%

For any $\beta \in (0,K)$, and conditional on $\mathcal{D}_n$ we define 
\begin{equation}
\label{eq:eqtldeSeuil}
\widetilde{G}^{-1}(\beta) = \inf\{t\in \mathbb{R} : \  \widetilde{G}(t) \leq \beta \}.
\end{equation}
We note that Assumption $(A2)$ ensures that $t \mapsto \widetilde{G}(t)$ is continuous and then
\begin{equation*}
\widetilde{G}(\widetilde{G}^{-1}(\beta)) = \beta.
\end{equation*}
Now, we have 
\begin{eqnarray*}
|\widetilde{\Gamma}_{\beta}(X)| & = & \sum_{k=1}^K \1_{\{\widetilde{G}(\hat{f}_k(X)) \leq \beta\}}\\
                            & = & \sum_{k=1}^K \1_{\{\hat{f}_k(X) \geq \widetilde{G}^{-1}(\beta)\}}.
\end{eqnarray*}
Hence, the last equation implies that
\begin{equation}
\label{eq:eqtldeCardSeuil}
\mathbb{E}_{X}\left[|\widetilde{\Gamma}_{\beta}(X)||\mathcal{D}_n\right] = \sum_{k = 1}^K \mathbb{P}_X\left(\hat{f}_k(X) \geq \widetilde{G}^{-1}(\beta)|\mathcal{D}_n\right) = \widetilde{G}
(\widetilde{G}^{-1}
(\beta)) = \beta.
\end{equation}
Therefore, we obtain $\mathbf{E}\left[|\widetilde{\Gamma}_{\beta}(X)|\right] = \beta$.
Also, we can write
\begin{eqnarray*}
\left|\mathbf{E}\left[\left|\hat{\Gamma}_{\beta}(X)\right|\right] - \beta \right| & \leq & \left|\mathbf{E}\left[|\hat{\Gamma}_{\beta}(X)| - |\widetilde{\Gamma}_{\beta}(X)|\right]\right|\\
                                                      & \leq &  \left|\mathbf{E}\left[\sum_{k = 1}^K \left(\1_{\{\hat{G}(\hat{f}_k(X)) \leq \beta\}}
                                                                         - \1_{\{\widetilde{G}(\hat{f}_k(X)) \leq \beta\}}\right)\right] \right|\\
                                                      & \leq & \mathbf{E}\left[|\hat{\Gamma}_{\beta}(X) \; \Delta \; \widetilde{\Gamma}_{\beta}(X)|\right] \\
                                                      & \leq & \sum_{k = 1}^K \mathbf{E}\left[|\1_{\{\hat{G}(\hat{f}_k(X)) \leq \beta\}}
                                                                         - \1_{\{\widetilde{G}(\hat{f}_k(X)) \leq \beta\}}|\right]\\
                                                      &\leq & \sum_{k=1}^K \mathbf{P}\left(|\hat{G}(\hat{f}_k(X) - \widetilde{G}(\hat{f}_k(X))| \geq |\widetilde{G}
                                                      (\hat{f}_k(X)) 
                                                      -\beta|\right).
\end{eqnarray*} 
Hence, applying Lemma~\ref{lem:lemFondamental} in the above inequality, we obtain the desired result.

\subsection{Proof of Theorem~\ref{theo:vitesse1}}
When there is no doubt, we suppress the dependence on $X$.
First, let us state a intermediate result that is also needed to prove the theorem.
\begin{lm}
\label{lm:decompRisk}
Consider $\Gamma_{\hat f,\delta}$ the confidence set based on the score function $\hat f$ with information $\beta$ (that is, $G_{\hat{f}}(-\delta) = \beta$).
Under assumptions $M_{\alpha}^k$, the following holds
\begin{equation*}
\Delta \mathcal R(\Gamma_{\hat{f},\delta}) \leq C(\alpha,s)\left\{ K^{1-1/(s +\lambda -\lambda s)} \Delta R_{\phi}(\hat{f})^{1/(s+\lambda -\lambda s)}
+  K^{1-\lambda/(s+\lambda -\lambda s)} \Delta R_{\phi}(\hat{f})^{\lambda/(s+\lambda -\lambda s)} \right\},
\end{equation*}
where $\lambda = \frac{\alpha}{\alpha+1}$ and $C(\alpha,s)$ is non negative constant which depends only on  $\alpha$ and $s$. 
\end{lm}
\begin{proof}
For each  $k = 1, \ldots, K$, we define the following events 
$S_k = B_k^c \cap  \{k\in (\Gamma_{\hat{f}, \delta} \, \Delta  \, \Gamma_\beta^{*})\}$ and $T_k = B_k \cap  \{k\in (\Gamma_{\hat{f}, \delta} \, \Delta  \, \Gamma_\beta^{*})\}$, where the $B_k$'s are the events given in Eq.~\eqref{eq:proofSetBk}.
Now we observe that
\begin{equation*}
\Delta \mathcal R(\Gamma_{\hat{f},\delta}) =\mathbb{E}\left[\sum_{k =1}^K \1_{S_k}  |p_k - G^{-1}(\beta)|\right] +  \mathbb{E}\left[\sum_{k =1}^K \1_{T_k}  |p_k - G^{-1}(\beta)|\right]
= \Delta \mathcal R_1(\Gamma_{\hat{f},\delta}) + \Delta  \mathcal R_2(\Gamma_{\hat{f},\delta}),
\end{equation*}
where
\begin{eqnarray*}
\Delta \mathcal R_1(\Gamma_{\hat{f},\delta}) & = & \mathbb{E}\left[\sum_{k =1}^K \1_{S_k}  |p_k - G^{-1}(\beta)|\right], \\
\Delta  \mathcal R_2(\Gamma_{\hat{f},\delta}) & = & \mathbb{E}\left[\sum_{k =1}^K \1_{T_k}  |p_k - G^{-1}(\beta)|\right].
\end{eqnarray*}
The end of the proof consists in controlling each of these two terms.
Let us first consider $\Delta \mathcal R_1(\Gamma_{\hat{f},\delta})$. For $\varepsilon > 0$, we have
\begin{eqnarray}
\label{eq:proofFinde}
\Delta  \mathcal R_1(\Gamma_{\hat{f},\delta}) & = & \mathbb{E}\left[\sum_{k = 1}^K \1_{S_k} |p_k - G^{-1}(\beta)| \left( \1_{\{|p_k - G^{-1}(\beta)| \geq \varepsilon\}}
+ \1_{\{|p_k - G^{-1}(\beta)| \leq \varepsilon\}} \right)\right] \nonumber \\
& \leq & \mathbb{E}\left[\sum_{k = 1}^K \1_{S_k} \varepsilon^{1-s} |p_k - G^{-1}(\beta)|^s\right] + \varepsilon \sum_{k=1}^K \mathbb{P}(S_k) \nonumber \\
& \leq & C \varepsilon^{1-s} \Delta R_{\phi}(\hat{f}) + \varepsilon \sum_{k=1}^K \mathbb{P}(S_k),
\end{eqnarray}
where  we used the assumption~\eqref{eq:Zhansgg} and more precisely~\eqref{eq:defout} to deduce the last inequality. 
To control $\sum_{k=1}^K \mathbb{P}(S_k)$,
we require the following result that is a direct application of Lemma~5 in~\cite{BartlettJordanMcauliffe06}.
%%%%
%%%%
\begin{lm}
\label{lm:margin}
Under the assumptions $M_{\alpha}^k$ we have
\begin{equation*}
\sum_{k =1}^K \mathbb{P}\left(S_k\right)  \leq
 C(\alpha) \left(K^{1/\alpha} \Delta \mathcal R_1(\Gamma_{\hat f,\delta}) \right)^{\frac{\alpha}{\alpha+1}},
\end{equation*}
 where $ C(\alpha) >0$ is a constant that depends only on $\alpha$.
\end{lm}
\begin{proof}
The proof of this result relies on the following simple fact: for all $\varepsilon > 0$ 
\begin{eqnarray*}
\mathbb{E}\left[\1_{S_k} |p_k - G^{-1}(\beta)|\right]  & \geq & 
\varepsilon \mathbb{E}\left[\1_{S_k} \1_{\{|p_k - G^{-1}(\beta)| \geq \varepsilon\}} \right]\\
& \geq & \varepsilon \left[\mathbb{P}\left(S_k\right) - c_2 \varepsilon^{\alpha} \right].
\end{eqnarray*}
Choosing $\varepsilon = \left(\frac{1}{c_2K(\alpha + 1)}\sum_{k =1}^K \mathbb{P}\left(S_k)\right) \right)^{1/\alpha}$ we get the lemma.
\end{proof}
%%%%%%%%
%%%%%%%%%
We go back to the proof of Lemma~\ref{lm:decompRisk}.
Applying Lemma~\ref{lm:margin} to~\eqref{eq:proofFinde}, we get
\begin{equation*}
\Delta \mathcal R_1(\Gamma_{\hat f,\delta}) \leq C\varepsilon^{1-s} \Delta R_{\phi}(\hat f) + \varepsilon C(\alpha) \left(K^{1/\alpha} \Delta \mathcal R_1(\Gamma_{\hat f,\delta}) \right)^{\frac{\alpha}
{\alpha+1}}.
\end{equation*}
Choosing $\varepsilon = \frac{s-1}{s  C(\alpha)} K^{(\lambda - 1)} \Delta \mathcal R_1(\Gamma_{\hat f,\delta})^{(1-\lambda)}$, we obtain
\begin{equation}
\label{eq:eqlmMajoration}
\Delta \mathcal R_1(\Gamma_{\hat f,\delta}) \leq  C_1(\alpha,  s ) K^{1-1/(s-\lambda s +\lambda)} \Delta R_{\phi}(\hat f)^{1/(s+\lambda -\lambda s)},
\end{equation}
for a non negative constant $C_1(\alpha,  s )$ that depends only on $\alpha$ and $  s$. 

Let us now focus on the second term, $ \Delta \mathcal R_2(\Gamma_{\hat f,\delta})$. 
Since the assumptions of Theorem~\ref{theo:theo1} are satisfied, we can use Eq.~\eqref{eq:eqMajoration} for any $\gamma > 0$.
Combined with the Margin assumptions $M_\alpha^k$, we obtain
\begin{eqnarray*}
\Delta \mathcal R_2(\Gamma_{\hat f,\delta}) & \leq & \mathbb{E}_{X}\left[|G_{f^{*}}(-\delta^{*}) -G_{{\hat f}}(-\delta^{*})|\right]\\
          & \leq & \frac{C\Delta R_{\phi}({\hat f})}{\gamma} + c_2K\gamma^{\alpha/s}.
\end{eqnarray*}
Therefore, optimizing in $\gamma$, we have
\begin{equation*}
\Delta \mathcal R_2(\Gamma_{\hat f,\delta}) \leq C_2(\alpha,s) K^{1-\lambda/(s+\lambda -\lambda s)} \Delta R_{\phi}(\hat f)^{\lambda/(s+\lambda - \lambda s)},
\end{equation*}
for a non negative constant $C_2(\alpha,  s )$ that depends only on $\alpha$ and $  s$.
The result stated in the lemma is deduced by combining the last equation with Eq.~\eqref{eq:eqlmMajoration} and by setting $C(\alpha,s)= \max\{ C_1(\alpha,s) ;C_2(\alpha,s)\}$.
\end{proof}
%%%%%%%%%%
%%%%%%%%%%
We now state another important lemma that describes the behavior of the empirical minimizer of the $\phi$-risk on the class $\mathcal F$:
\begin{lm}
\label{lm:estimPhiRisk}
Let $\bar{f} \in \mathcal{F}$ be the minimizer of $R_{\phi}(f)$ over $\mathcal{F}$.
Under the assumptions of Theorem~\ref{theo:vitesse1}, we have that
\begin{equation*}
\mathbf{E}\left[R_{\phi}(\hat{f_n}) -R_{\phi}(\bar{f})\right] \leq \frac{3KL}{n} + \frac{K C(B,L) \log(N_n)}{n},
\end{equation*}
where $C(B,L)>0$ is a constant that only depends on the constant $B$ given in the Proposition~\ref{prp:crossVal} and on the Lipschitz constant $L$.
\end{lm}
\begin{proof}
%Now, we go back to the proof of theorem
%First, we have the following decomposition
%\begin{equation*}
%| R(\hat{\Gamma}_{\beta}) - R_{\beta}^* |\leq \Delta R(\widetilde{\Gamma}_{\beta})
%+ |R(\hat{\Gamma}_{\beta})-R(\widetilde{\Gamma}_{\beta})|.
%\end{equation*}
%
%
%
%Therefore, we have
%\begin{equation}
%\label{eq:eqTilde}
%\Delta R(\widetilde{\Gamma}_{\beta}) \leq C_3 \left( \Delta R_{\phi}(f)^{1/(s+\beta +\beta s)} + \Delta R_{\phi}(f)^{s/(\alpha+s)}\right)
%\end{equation}
%
%The following step of the proof is to bound the estimation error of the $\phi$-risk over the family $\mathcal{F}$
First, according to Eq.~\eqref{eq:deltaMod}, for each $f \in \mathcal{F}$, we can write
\begin{equation*}
\frac{R_{\phi}(f)+R_{\phi}(\bar{f})}{2} - R_{\phi}\left(\frac{f+\bar{f}}{2}\right) \geq \delta\left(\sqrt{\sum_{k=1}^K\mathbb{E}_{X}\left[(f_k-\bar{f}_k)^2(X)\right]}\right),
\end{equation*}
Hence, by assumption on the modulus of convexity, we deduce
\begin{equation*}
\frac{R_{\phi}(f)+R_{\phi}(\bar{f})}{2} - R_{\phi}\left(\frac{f+\bar{f}}{2}\right)  \geq   c_1 \sum_{k=1}^K\mathbb{E}_{X}\left[(f_k-\bar{f}_k)^2(X)\right]. 
\end{equation*}
Since $R_{\phi}\left(\frac{f+\bar{f}}{2}\right) \geq R_{\phi}(\bar{f})$, we obtain
\begin{equation*}
\sum_{k=1}^K\mathbb{E}_{X}\left[(f_k-\bar{f}_k)^2(X)\right] \leq \frac{1}{2c_1} R_{\phi}(f) - R_{\phi}(\bar{f}).
\end{equation*}
Now, denoting by $h(z,f(x)) = \sum_{k = 1}^ K \phi(z_kf_k(x)) - \phi(z_k\bar{f}_k(x))$, we get the following bound
\begin{eqnarray}
\label{eq:eqvarB}
\mathbb{E}_X\left[h^2(Zf(X))\right]  & \leq & K L^2 \sum_{k=1}^K\mathbb{E}_{X}\left[(f_k-\bar{f}_k)^2(X)\right]  \nonumber \\ 
                                     & \leq & \frac{K L^2}{2c_1} \mathbb{E}_X\left[h(Zf(X))\right],
\end{eqnarray}
where $L$ is the Lipschitz constant $L$ for $\phi$.
On the other hand, we have the following decomposition
\begin{equation*}
R_{\phi}(\hat{f}) - R_{\phi}(\bar{f}) = R_{\phi}(\hat{f}) +  2(\hat{R}_{\phi}(\hat{f}) - \hat{R}_{\phi}(\bar{f})) - 2 (\hat{R}_{\phi}(\hat{f}) - \hat{R}_{\phi}(\bar{f}))
- R_{\phi}(\bar{f}).
\end{equation*}
Also, since $\hat{R}_{\phi}(\hat{f}) - \hat{R}_{\phi}(\bar{f}) \leq 0$, we get
\begin{eqnarray*}
R_{\phi}(\hat{f}) - R_{\phi}(\bar{f}) & \leq &  (R_{\phi}(\hat{f}) - R_{\phi}(\bar{f})) - 2 (\hat{R}_{\phi}(\hat{f}) - \hat{R}_{\phi}(\bar{f})) \\
                                      & \leq & \frac{3KL}{n} +  \sup_{f \in \mathcal{F}_n}  (R_{\phi}({f}) - R_{\phi}(\bar{f})) - 2 (\hat{R}_{\phi}({f}) - \hat{R}_{\phi}(\bar{f})),
\end{eqnarray*}
where $\mathcal{F}_n$ is the $\epsilon$-net of $\mathcal{F}$ w.r.t the $L_\infty$-norm and with $\epsilon = 1/n$.
Now, using Bernstein's Inequality, we have that for all $f \in \mathcal{F}_n$ and $t > 0$
\begin{multline*}
\mathbf{P}\left((R_{\phi}({f}) - R_{\phi}(\bar{f})) - 2 (\hat{R}_{\phi}({f}) - \hat{R}_{\phi}(\bar{f})) \geq t\right)  \leq \\
\mathbf{P}\left(2((R_{\phi}({f}) - R_{\phi}(\bar{f})) - (\hat{R}_{\phi}({f}) -\hat{R}_{\phi}(\bar{f})) \geq t + R_{\phi}({f}) - R_{\phi}(\bar{f})\right)\\
 \leq  \exp\left(-\dfrac{n(t+\mathbb{E}\left[h(Z,f(X))\right])^2/8}{\mathbb{E}\left[h^2(Z,f(X))\right])+(2KLB/3)(t+\mathbb{E}\left[h(Z,f(X))\right]))}\right).              
\end{multline*}
Using Eq.~\eqref{eq:eqvarB}, we get for all $f \in \mathcal{F}_n$
\begin{equation*}
\mathbf{P}\left((R_{\phi}({f}) - R_{\phi}(\bar{f})) - 2 (\hat{R}_{\phi}({f}) - \hat{R}_{\phi}(\bar{f})) \geq t\right)
\leq \exp\left(-\frac{nt}{8(KL^2/(2c_1) +  KLB/3)}\right),
\end{equation*}
Therefore, using a union bound argument, and then integrating we deduce that
\begin{eqnarray*}
\mathbf{E}\left[R_{\phi}(\hat{f}) - R_{\phi}(\bar{f})\right] &\leq & \frac{3KL}{n} +  \mathbf{E}\left[\sup_{f \in \mathcal{F}_n} (R_{\phi}({f}) - R_{\phi}(\bar{f})) - 2 
(\hat{R}_{\phi}({f}) - \hat{R}_{\phi}(\bar{f}))\right]\\
& \leq & \frac{3KL}{n} + \frac{K C( B,L)  \log(M_n)}{n}.
\end{eqnarray*}
%
%To conclude the proof it remains to note that $\Delta R_{\phi}(\hat{f}) \leq \inf_{f\in\mathcal{F}} \Delta R_{\phi}(f) + \frac{K C_4 \log(N_n)}{n}$.
\end{proof}

We are now ready to conclude the proof of the theorem.
We have the following inequality
\begin{equation}
\label{eq:eqDecomp}
| \mathcal R(\hat{\Gamma}_{\beta}) - \mathcal R_{\beta}^* |\leq \Delta \mathcal R(\widetilde{\Gamma}_{\beta})
+ |\mathcal R(\hat{\Gamma}_{\beta})-\mathcal R(\widetilde{\Gamma}_{\beta})|.
\end{equation}
We deal which each terms in the r.h.s separately.
First, we have from Jensen's Inequality that
\begin{equation*}
\left(\mathbf{E}\left[\Delta \mathcal R(\widetilde{\Gamma}_{\beta})\right]\right)^{\frac{s+\lambda-\lambda s}{\lambda}}
\leq \mathbf{E}\left[\Delta\mathcal R(\widetilde{\Gamma}_{\beta})^{\frac{s+\lambda-\lambda s}{\lambda}}\right].
\end{equation*}
Hence, from Lemma~\ref{lm:decompRisk}, we deduce
\begin{equation*}
\left(\mathbf{E}\left[\Delta \mathcal R(\widetilde{\Gamma}_{\beta})\right]\right)^{\frac{s+\lambda-\lambda s}{\lambda}}
\leq C(\alpha,s)^{\frac{s+\lambda-\lambda s}{\lambda}}K^{\frac{s+\lambda-\lambda s}{\lambda}-1} 
\mathbf{E}\left[\Delta R_{\phi}(\hat{f})\right].
\end{equation*}
Moreover, from Lemma~\ref{lm:estimPhiRisk}, we have that
\begin{equation*}
\mathbf{E}\left[\Delta R_{\phi}(\hat{f})\right] \leq \inf_{f \in \mathcal{F}}\Delta R_{\phi}(f) + 
\frac{3KL}{n} + \frac{K C( B,L) \log(M_n)}{n}.
\end{equation*}
Therefore, we can write
\begin{equation}
\label{eq:eqIneq1}
\mathbf{E}\left[\Delta \mathcal{R}(\widetilde{\Gamma}_{\beta})\right] \leq C(\alpha,s) K^{1-\lambda/(s+\lambda -\lambda s)}
\left\{ \inf_{f \in \mathcal{F}}\Delta R_{\phi}(f) + \frac{3KL}{n} + \frac{K C( B,L) \log(N_n)}{n} \right\}^{\lambda/(s+\lambda -\lambda s)}.
\end{equation}
For the second term $|\mathcal R(\hat{\Gamma}_{\beta})-\mathcal R(\widetilde{\Gamma}_{\beta})|$ in~\eqref{eq:eqDecomp}, we observe that
\begin{equation*}
 	\1_{\{Y \notin \hat{\Gamma}_\beta(X) \}} - \1_{\{Y \notin \widetilde{\Gamma}_\beta(X) \}} 
	 = 
	\sum_{k =1}^{K} \1_{ \left\{ Y = k\right\} }   \1_{ \left\{k \notin \hat{\Gamma}_\beta (X)\right\} } - \sum_{k =1}^{K}  
 \1_{ \left\{ Y = k  \right\} }   \1_{ \left\{ k \notin \widetilde{\Gamma}_\beta(X)  \right\}}. 
\end{equation*}
Therefore, we can write
\begin{eqnarray*}
\mathbf{E}\left[\1_{Y \notin \hat{\Gamma}_\beta(X) \}} - \1_{\{Y \notin \widetilde{\Gamma}_\beta(X) \}} \right]
& = & \sum_{k = 1}^K \mathbf{E}\left[{p}_k(X)\left( \1_{ \left\{k \notin \hat{\Gamma}_\beta(X)\right\} } - \1_{\{k \notin \widetilde{\Gamma}_\beta(X)\}}\right)\right]\\
& = & \sum_{k = 1}^K \mathbf{E}\left[{p}_k(X)\left( \1_{ \left\{\hat{G}(\hat{f}_k(X)) > \beta\right\} } - \1_{\{\tilde{G}(\hat{f}_k(X)) > \beta\}}\right)\right].
\end{eqnarray*}
Since $0\leq {p}_k (X) \leq 1$ for all $k \in \{1,\ldots,K\}$, the last equality implies
\begin{multline*}
\left|\mathbf{R}\left(\hat{\Gamma}_{\beta}\right) - \mathbf{R}\left(\tilde{\Gamma}_{\beta}\right)\right| =  \left|\mathbf{E}\left[\1_{Y \notin \hat{\Gamma}_\beta(X) \}} - 
\1_{\{Y \notin \widetilde{\Gamma}_\beta(X) \}} \right]\right| \\ \leq 
\sum_{k = 1}^K \mathbf{E}\left[\left|\1_{ \left\{\hat{G}(\hat{f}_k(X)) > \beta\right\} } - \1_{\{\tilde{G}(\hat{f}_k(X)) > \beta\}}\right|\right] \\ \leq 
\sum_{k=1}^K \mathbf{P}\left(|\hat{G}(\hat{f}_k(X) - \widetilde{G}(\hat{f}_k(X))| \geq |\widetilde{G}(\hat{f}_k(X)) -\beta|\right).
\end{multline*}
Therefore, Lemma~\ref{lem:lemFondamental} implies
\begin{equation}
\label{eq:eqIneq2}
\left|\mathbf{R}\left(\hat{\Gamma}_{\beta}\right) - \mathbf{R}\left(\tilde{\Gamma}_{\beta}\right)\right| \leq \dfrac{C' K}{\sqrt{N}}.
\end{equation}
Injecting Eqs.~\eqref{eq:eqIneq1} and~\eqref{eq:eqIneq2} to Eq.~\eqref{eq:eqDecomp}
we conclude the proof of the theorem.

\subsection{Proof of Proposition~\ref{prp:crossVal}}

We begin with the following decomposition
\begin{equation*}
\tilde{R}^n_{\phi}(\hat{f}) - \tilde{R}^n_{\phi}(\tilde{f}) = \tilde{R}^n_{\phi}(\hat{f}) +  2(\hat{R}^n_{\phi}(\hat{f}) - \hat{R}^n_{\phi}(\tilde{f})) - 2 (\hat{R}^n_{\phi}(\hat{f}) - \hat{R}^n_{\phi}(\tilde{f}))
- \tilde{R}^n_{\phi}(\tilde{f}),
\end{equation*}
since $\hat{R}^n_{\phi}(\hat{f}) - \hat{R}^n_{\phi}(\tilde{f}) \leq 0$, we get
\begin{equation}
\label{eq:eqCrossValDecomp}
\tilde{R}^n_{\phi}(\hat{f}) - \tilde{R}^n_{\phi}(\tilde{f})  \leq   (\tilde{R}^n_{\phi}(\hat{f}) - \tilde{R}^n_{\phi}(\tilde{f})) - 2 (\hat{R}^n_{\phi}(\hat{f}) - \hat{R}^n_{\phi}(\tilde{f})).
\end{equation} 
Now, we denote by $\mathcal{C}_n  = \{(i_1/n, \ldots, i_M/n), (i_1,\ldots ,i_M) \in \{0, \ldots ,n\}\} \cap \mathcal{C}_M$,\\
and $\mathcal{F}_n = \{ f = \sum_{i = 1}^M m_i f_i, \;\; m_1,\ldots m_M) \in \mathcal{C}_n\}$.
For each $f \in {\rm conv}(\mathcal{F})$, there exists $f_n \in \mathcal{F}_n$ such that
\begin{eqnarray*}
|\tilde{R}^n_{\phi}(f)- \tilde{R}^n_{\phi}(f_n)| & \leq &  \frac{2KLBM}{n} \\
|\hat{R}^n_{\phi}(f)- \hat{R}^n_{\phi}(f_n)| & \leq &   \frac{2KLBM}{n}.
\end{eqnarray*}
Therefore, with Equation~\eqref{eq:eqCrossValDecomp}, we obtain
\begin{equation*}
\tilde{R}^n_{\phi}(\hat{f}) - \tilde{R}^n_{\phi}(\tilde{f}) \leq \frac{6LKB}{n} +  \sup_{f \in \mathcal{F}_n}  (\tilde{R}^n_{\phi}({f}) - \tilde{R}^n_{\phi}(\bar{f})) - 2 (\hat{R}^n_{\phi}({f}) - \hat{R}^n_{\phi}(\bar{f})).
\end{equation*}
Now, Similar arguments as in~\cite{DudoitvdLaan} and Lemma~\ref{lm:estimPhiRisk} yield the proposition.

\bibliographystyle{alpha}      % basic style, author-year citations
\bibliography{ConfSet}

\end{document}